\documentclass[11pt]{amsart}

\def\pr{\partial}

\newcommand{\ddt}{\frac{\rm d}{{\rm d} t} }

\if {

}{{\caption}}
} \fi

\newcommand\findem{\hfill{$\square$}\medskip}

\def\uh{\hat{u}}
\def\vh{\hat{v}}
\def\wh{\hat{w}}
\def\xh{\hat{x}}
\def\yh{\hat{y}}

\def\hb{\bar{h}}

\def\ub{\bar{u}}
\def\vb{\bar{v}}
\def\wb{\bar{w}}
\def\xb{\bar{x}}
\def\yb{\bar{y}}

\def\xib{\bar{\xi}}

\def\C{\mathcal{C}}

\def\H{\mathcal{H}}

\def\P{\mathcal{P}}
\def\Q{\mathcal{Q}}
\def\R{\mathcal{R}}

\def\U{\mathcal{U}}
\def\V{\mathcal{V}}
\def\W{\mathcal{W}}
\def\X{\mathcal{X}}
\def\Y{\mathcal{Y}}

\def\eps{\varepsilon}

\def\half{\mbox{$\frac{1}{2}$}}

\def\1B{{\bf  1}}

\newcommand{\cR}{\mathbb{R}}

\newcommand\be{\begin{equation}}
\newcommand\ee{\end{equation}}
\newcommand{\benl}{\begin{equation*}}
\newcommand{\eenl}{\end{equation*}}
\newcommand\ba{\begin{array}}
\newcommand\ea{\end{array}}
\newcommand{\bean}{\begin{eqnarray*}}
\newcommand{\eean}{\end{eqnarray*}}

\def\ds{\displaystyle}

\newcommand{\tras}{^\top}
\newcommand{\intT}{\int_0^T }

\newcommand{\dtt}{\mathrm{d}t}
\newcommand{\mr}{\mathrm}

\usepackage{amsmath,amssymb,latexsym,amsthm,mathdots,yhmath,xcolor}

\newtheorem{theorem}{Theorem}[section]
\newtheorem{lemma}[theorem]{Lemma}
\newtheorem{proposition}[theorem]{Proposition}
\newtheorem{corollary}[theorem]{Corollary}

\theoremstyle{remark}
{
    \newtheorem{definition}[theorem]{Definition}
    
    \newtheorem{remark}[theorem]{Remark}

\newtheorem{assumption}[theorem]{Assumption}
}

\DeclareMathAlphabet{\mathpzc}{OT1}{pzc}{m}{it}

\newcommand{\gr}{>}
\newcommand{\mi}{<}

\newcommand{\lam}{[\lambda]}
\newcommand{\lamh}{[\hat\lambda]}

\newcommand{\cd}{(\cdot)}
 \usepackage[colorlinks=true]{hyperref}
 

\begin{document}

\title[Second order analysis of partially-affine control problems]{
Second order necessary and sufficient optimality conditions for singular solutions of partially-affine control problems}

\footnotetext{This article has been accepted for publication in Discrete Contin. Dyn. Syst. Ser. S.}

\author[M.S. Aronna]{M. Soledad Aronna}
\address{M.S. Aronna\\ Escola de Matem\'atica Aplicada, Funda\c c\~ ao Getulio Vargas, Praia de Botafogo 190,   22250-900 Rio de Janeiro - RJ, Brazil}
\email{soledad.aronna@fgv.br}


\maketitle



\begin{abstract} 
In this article we study optimal control problems for systems that are affine with respect to some of the control variables and nonlinear in relation to the others. We consider finitely many equality and inequality constraints on the initial and final values of the state.
We investigate singular optimal solutions for this class of problems, for which we obtain second order necessary and sufficient conditions for weak optimality in integral form. We also derive Goh pointwise necessary optimality conditions. 
We show an example to illustrate the results. 

\end{abstract}

\section{Introduction}\label{Introduction}

The purpose of this paper is to investigate optimal control problems governed by systems of ordinary differential equations of the form
\benl
\dot{x}=f_0(x,u)+ \sum_{i=1}^m v_{i} f_i(x,u),\quad  {\rm  a.e.}\ {\rm  on}\ [0,T].
\eenl
Here $x:[0,T]\to \cR^n$ is the state variable,  $v_i:[0,T]\to \cR$ are the {\em affine} controls for $i=1,\dots m,$ while $u:[0,T]\to \cR^l$ is the vector of {\em nonlinear} controls and $f_i:\cR^{n+l}\to \cR^n$ is a vector field, for each $i=0,\dots m.$ 

Many models that enter into this framework can be found in practice and, in particular, in the existing literature. Among these we can mention:  the Goddard's problem in three dimensions  \cite{Goddard}  analyzed in Bonnans et al. \cite{BLMT09}, several models concerning the motion of rockets as the ones treated in Lawden \cite{Law63}, Bell and Jacobson \cite{BelJac}, Goh \cite{GohThesis,Goh08}, Oberle \cite{Obe77}, Azimov \cite{Azi05} and Hull \cite{Hul11}; an hydrothermal electricity production problem studied in Bortolossi et al. \cite{BPT02}, the problem of atmospheric flight considered by Oberle in \cite{Obe90}, and the optimal production processes studied in Cho et al. \cite{ChoAbadParlar93} and Maurer at al. \cite{MauKimVos05}. All the systems investigated in these cited articles are {\em partially-affine} in the sense that they have at least one affine and at least one nonlinear control.

The subject of second order optimality conditions for these partially-affine problems has been studied by Goh in \cite{GohThesis,Goh66a,Goh67,Goh08}, Dmitruk in \cite{Dmi11}, Dmitruk and Shishov in \cite{DmiShi10}, Bernstein and Zeidan \cite{BerZei90}, Frankowska and Tonon \cite{FraTon13}, and Maurer and Osmolovskii \cite{MauOsm09}.
The first works were by Goh, who introduced a change of variables in \cite{Goh66a} and used it to obtain necessary optimality conditions in \cite{Goh66a,GohThesis,Goh66}, always assuming {\em normality} of the optimal solution. The necessary conditions we present imply those by Goh \cite{Goh66}, when there is only one multiplier (see Corollary \ref{CoroCBsym}).
Recently, Dmitruk and Shishov \cite{DmiShi10} analyzed the  quadratic functional associated with the second variation of the Lagrangian function, and provided a set of necessary conditions for the nonnegativity of this quadratic functional. Their results are consequence of a second order necessary condition that we present (see Theorem \ref{NCP2}).
In  \cite{Dmi11}, Dmitruk proposed, without proof, necessary and sufficient conditions for a problem having a particular structure: the affine control variable applies to a term depending only on the state variable, i.e. the affine and nonlinear controls are  {\em uncoupled} or, equivalently  $H_{uv}$ is identically zero, where $H$ denotes the {\em unmaximized Hamiltonian.} 
This hypothesis is not used in our work.
Nevertheless, the conditions established here coincide with those suggested in Dmitruk \cite{Dmi11}, when the latter are applicable.
In \cite{BerZei90}, Bernstein and Zeidan derived the Riccati equation for the {\em singular linear-quadratic regulator,} which is a modification of the classical linear-quadratic regulator where only some components of the control enter quadratically in the cost function.
Frankowska and Tonon proved in \cite{FraTon13} second order necessary conditions for problems with {\em closed} control constraints  and optimal controls containing arcs along which the second order derivative $H_{uu}$ of the unmaximized Hamiltonian vanishes.  The necessary conditions given in \cite{FraTon13} hold for problems either  with  no endpoint constraints, or with smooth endpoint constraints and additional hypotheses as {\em calmness} and the {\em abnormality} of  Pontryagin's Maximum Principle.
All the articles mentioned in this paragraph use {\em Goh's transformation} to derive their optimality conditions, as it is done in the current paper, while none of them proved sufficient conditions of second order which is the main contribution of this article.
It is worth mentioning that sufficient conditions were shown by Maurer and Osmolovskii in \cite{MauOsm09}, but for the case of a scalar control subject to bounds and {\em bang-bang} optimal solutions (i.e. no singular arc).  This structure is not studied here since no closed control constraints are considered and thus our optimal control is supposed to be {\em singular} along the whole interval.

 The contributions of this article are as follows.
We provide a pair of necessary and sufficient conditions in integral form for weak optimality of singular solutions of partially-affine problems (Theorems \ref{NCP2}-\ref{SC}). 
These conditions  are {\em `no gap'} in the sense that  the sufficient condition is obtained from the necessary one by strengthening an inequality. We consider fairly general endpoint constraints and we do not assume uniqueness of multiplier.
The main result is the sufficient condition of Theorem \ref{SC}, which, up to our knowledge, cannot been found in the existing literature, and has important practical applications. 
As a product of the necessary condition \ref{NCP2} we get the {\em pointwise Goh conditions} in Corollary \ref{CoroCBsym}, extending this way previous results (see \cite{Goh66,FraTon13}) to problems with general endpoint constraints, and removing the hypothesis of vanishing $H_{uu}$ imposed in \cite{FraTon13}. 
 In order to obtain the sufficient condition we impose a regularity assumption on the optimal controls, that in some practical situations is a consequence of the {\em generalized Legendre-Clebsch condition} (see Remark \ref{RemarkLC}). We provide a simple example to illustrate our results.


As a main application of the sufficient condition provided in this article we can mention the proof of convergence of an associated {\em shooting algorithm} as stated in Aronna \cite{Aro13} and shown in detail in the technical report Aronna \cite{Aro11}. It is worth mentioning that, for practical interest, this shooting algorithm and its proof of convergence can be also used to solve partially-affine problems with bounds on the control and associated bang-singular solutions.



The article is organized as follows. In Section \ref{SectionPb} we present the problem, the basic definitions and first order optimality conditions.
In Section \ref{SectionSOC} we give the tools for second order analysis and establish a second order necessary condition. We introduce Goh's transformation in Section \ref{GohT}.
In Section \ref{SectionNC} we show a new second order necessary condition. In Section \ref{SectionSC} we present the main result of this article that is a second order sufficient condition. We show an example to illustrate our results in Section \ref{SectionExample}, while Section \ref{SectionConclusion} is devoted to the conclusions and possible extensions. Finally, we include an Appendix containing some proofs of technical results that are omitted throughout the article.

\vspace{8pt}

\noindent\textbf{Notations.}
Given a function $h$ of variable $(t,x)$, we write $D_th$ or $\dot{h}$ for its derivative in time, and $D_xh$ or $h_x$ for the differentiations with respect to space variables. The same convention  is extended to higher order derivatives.
We let $\cR^k$ denote the $k$-dimensional real space, i.e. the space of column real vectors of dimension $k;$ and by  $\cR^{k,*}$ its corresponding dual space, which consists of $k-$dimensional real row vectors.
By $L^p(0,T;\cR^k)$ we mean the Lebesgue space with domain equal to the interval $[0,T]\subset \cR$ and with values in $\cR^k.$ The notation $W^{q,s}(0,T;\cR^k)$ refers to the Sobolev spaces (see e.g. Adams \cite{Ada75}). Given $A$ and $B$ two $k\times k$ symmetric real matrices, we write $A\succeq B$ to indicate that $A-B$ is positive semidefinite.
Given two functions $k_1:\cR^N \rightarrow \cR^{M}$ and $k_2: \cR^N  \rightarrow \cR^L,$ we say that $k_1$ is a {\it big-O} of $k_2$ around 0 and write
\benl
k_1(x) = \mathcal{O} (k_2(x)),
\eenl
if there exists positive constants $\delta$ and $M$ such that $|k_1(x)| \leq M|k_2(x)|$ for $|x|<\delta.$ It is a {\it small-o} if $M$ goes to 0 as $|x|$ goes to 0, and in this case we write
\benl
k_1(x) = o(k_2(x)).
\eenl

\section{Statement of the problem and assumptions}\label{SectionPb}

\subsection{Statement of the problem.} 
We study the optimal control
problem (P) given by
\begin{align}
\min\,\,&\label{cost} \varphi_0(x(0),x(T)),\\
&\label{stateeq}\dot{x}=F(x,u,v),\quad {\rm  a.e.}\ {\rm  on}\ [0,T],\\
& \label{finaleq} \eta_j(x(0),x(T))=0,\quad
\mathrm{for}\
j=1\hdots,d_{\eta},\\
&\label{finalineq} \varphi_i(x(0),x(T))\leq 0,\quad
\mathrm{for}\
i=1,\hdots,d_{\varphi},\\
&\label{UV} u(t)\in U ,\,\, v(t)\in V,\quad {\rm  a.e.}\ {\rm  on}\ [0,T],
\end{align}
where the function $F\colon \cR^{n+l+m}\to\cR^n$ can be written as
\benl
F(x,u,v):=f_0(x,u) + \sum_{i=1}^m v_{i} f_i(x,u).
\eenl
Here  $f_i \colon \cR^{n+l}\rightarrow
\cR^n$ for $i=0,\hdots,m,$ 
$\varphi_i \colon  \cR^{2n}\rightarrow \cR$ for
$i=0,\hdots,d_{\varphi},$
$\eta_j \colon  \cR^{2n}\rightarrow \cR$ for
$j=1,\hdots,d_{\eta}.$  
The sets $U$ and $V$ are open domains of $\cR^l$ and $\cR^m,$ respectively. The control $u\cd$ is called {\em nonlinear,} while $v\cd$ is named {\em affine control.} 
We consider the function spaces $\U:=L^{\infty}(0,T;\cR^l)$ and   $\V:=L^{\infty}(0,T;\mathbb{R}^m)$  for the controls, and $\mathcal{X}:=W^{1,\infty}(0,T;\mathbb{R}^n)$ for the state. 
 When needed, we use $w\cd:=(x,u,v)\cd$ to refer to a point in
$\mathcal{W}:=\mathcal{X}\times \U \times \V.$ 
We call {\em trajectory} an element $w\cd\in\mathcal{W}$ that satisfies the state equation \eqref{stateeq}. If in addition, the endpoint constraints \eqref{finaleq} and
\eqref{finalineq} and the control constraint \eqref{UV} hold for $w\cd,$ then we say that it is a \textit{feasible trajectory} of problem (P). 


\vspace{5pt}

We consider the following regularity hypothesis throughout the article.

\begin{assumption} 
\label{regular}
All data functions have Lipschitz-continuous second order derivatives.
\end{assumption}

In this paper we study optimality conditions for {\em weak minima} of problem (P). A feasible trajectory $\wh\cd=(\xh,\uh,\vh)\cd$ is said to be a {\em weak minimum} if there exists $\eps>0$ such that the cost function attains at $\wh\cd$ its minimum in the set of feasible
trajectories $w\cd=(x,u,v)\cd$ satisfying
\benl
\|x-{\xh}\|_{\infty}<\varepsilon,\quad
\|u-\uh\|_{\infty}<\varepsilon,\quad 
\|v-\vh\|_{\infty}\mi \eps.
\eenl

For the remainder of the article, we fix a nominal feasible trajectory $\wh\cd:=(\xh,\uh,\vh)\cd$ for which we provide optimality conditions. We assume that the controls $\uh\cd$ and $\vh\cd$ do not accumulate at the boundaries of $U$ and $V,$ respectively. This is, letting $\mathbb{B}$ denote the closed unit ball of $\cR^{l+m},$ we impose:
\begin{assumption} 
\label{uvboundary}
There exists $\delta>0$ such that $(\uh,\vh)(t)+\delta \mathbb{B} \subset U\times V,$ for almost all $t\in [0,T].$
\end{assumption}

An element $\delta w\cd\in \W$ is termed \textit{feasible variation for $\wh\cd$} if $\wh\cd+\delta
w\cd$ is a feasible trajectory for (P). 
For $\lambda=(\alpha,\beta,p(\cdot))$ in the space
$\cR^{d_{\varphi}+1,*}\times \cR^{d_{\eta},*}\times
W^{1,\infty}(0,T;\cR^{n,*}),$
we define the following functions:
\begin{itemize}
 \item the \textit{pre-Hamiltonian (or unmaximized Hamiltonian)} function $H[\lambda]\colon \cR^n\times \cR^m\times \cR^l\times [0,T]\to \cR$ given by
\benl
H[\lambda](x,u,v,t):=p(t)\left(f_0(x,u)+\sum_{i=1}^m v_i
f_i(x,u)\right),
\eenl
\item the \textit{endpoint Lagrangian} function $\ell[\lambda]\colon \cR^{2n}\to \cR,$
\benl
\ell[\lambda](x_0,x_T):=\sum_{i=0}^{d_{\varphi}}
\alpha_i\varphi_i(x_0,x_T)+\sum_{j=1}^{d_{\eta}}\beta_j
\eta_j(x_0,x_T),
\eenl
\item and the \textit{Lagrangian function} $\mathcal{L}[\lambda]\colon \W\to \cR,$
\be\label{lagrangian}
\mathcal{L}[\lambda](w):= \ell[\lambda](x(0),x(T)) 
+ \int_0^{T}
p \left(f_0(x,u)+\sum_{i=1}^{m}v_{i}f_i(x,u)-\dot
x\right)\dtt.
\ee
\end{itemize}
We assume, in sake of simplicity of notation that, whenever some argument of $F,$ $f_i,$ $H,$ $\ell,$ $\mathcal{L}$ or their derivatives is omitted, they are evaluated at $\wh\cd.$ 
If we further want to explicit that they are evaluated at time $t,$ we write $F[t],$ $f_i[t],$ etc. The same convention notations hold for other functions of the state, control  and multiplier that we define throughout the article.
We assume, without any loss of generality, that
$$
\varphi_i(\xh(0),\xh(T))=0,\ \mr{for}\ \mr{all}\ 
i=1,\hdots,d_{\varphi}.
$$


\subsection{Lagrange multipliers}
We introduce here the concept of {\em multiplier.} The second order conditions that we prove in this article are expressed in terms of the second variation of the Lagrangian function $\mathcal{L}$ given in \eqref{lagrangian} and the {set of Lagrange multipliers} associated with $\wh\cd$ that we define below.

\begin{definition}
\label{DefMul}
An element $\lambda=(\alpha,\beta,p(\cdot))\in
\cR^{d_{\varphi}+1,*}\times \cR^{d_{\eta},*}\times
W^{1,\infty}(0,T;\cR^{n,*})$ is a \textit{Lagrange multiplier}
associated with $\wh\cd$ if it satisfies the following conditions:
\begin{align}
\label{nontriv}&|\alpha|+|\beta|=1,\\
&\label{alphapos}\alpha=(\alpha_0,\alpha_1,\hdots,\alpha_{d_{\varphi}})\geq0,
\end{align}
the function $p(\cdot)$ is solution of the \textit{costate
equation} 
\be
\label{costateeq}
-\dot{p}(t)=H_x[\lambda](\xh(t),\uh(t),\vh(t),t),
\ee
it satisfies the \textit{transversality conditions}
\be
\label{transvcond}
\begin{split}
p(0)&=-D_{x_0}\ell\lam(\xh(0),\xh(T)),\\
p(T)&=D_{x_T}\ell\lam(\xh(0),\xh(T)),
\end{split}
\ee
and the \textit{stationarity conditions} 
\be
\label{stationarity}
\left\{
\ba{l}
\vspace{3pt} \ds  H_u\lam(\xh(t),\uh(t),\vh(t),t)=0,\\
H_v\lam(\xh(t),\uh(t),\vh(t),t)=0,
\ea
\right.
\quad {\rm a.e.}\ {\rm on}\ [0,T].
\ee
 We let $\Lambda$ denote the {\it set of Lagrange multipliers} associated with $\wh\cd.$
\end{definition}

The following result constitutes a {\em first order necessary condition} and yields the existence of Lagrange multipliers.
\begin{theorem}
\label{LambdaCompact}
If $\wh\cd$ is a weak minimum for (P), then the set $\Lambda$ is non empty and compact.
\end{theorem}

\begin{proof}
The existence of a Lagrange multiplier follows from  Milyutin-Osmolovskii \cite[Thm. 2.1]{MilOsm98} or equivalent results proved in Alekseev et al.  \cite{AleTikFom79} and Kurcyusz-Zowe \cite{KurZow}. 
In order to prove the compactness, observe that $\Lambda$ is closed and that
$p(\cdot)$ may be expressed as a linear continuous
mapping of $(\alpha,\beta).$ Thus, since the
normalization \eqref{nontriv} holds, $\Lambda$ is necessarily a 
finite-dimensional compact set.
\end{proof}

In view of previous Theorem \ref{LambdaCompact}, note that  $\Lambda$ can be
identified with a compact subset of $\cR^s,$ where
$s:=d_{\varphi}+d_{\eta}+1.$
The main results of this article are stated on a restricted subset  of $\Lambda$ for which the matrix $D^2_{(u,v)^2} H\lam (\wh,t)$ is singular and, consequently, the pairs $(\wh,\lambda)$ result  to be {\em singular extremals}.  We comment again on this fact in Remark \ref{RemSing} below.

\vspace{4pt}

Given $(\xb_0,\ub(\cdot),\vb(\cdot))\in \cR^n\times \U\times \V,$ consider 
the \textit{linearized state equation}
\begin{align} 
\label{lineareq} 
\dot{\xb} &= F_{x}\,\xb + F_{u}\,\ub + F_{v}\,\vb,\quad {\rm a.e.}\  {\rm on}\ [0,T],\\
\label{lineareq0}
\xb(0) &= \xb_0.
\end{align}
The solution $\xb\cd$ of \eqref{lineareq}-\eqref{lineareq0} is called
\textit{linearized state variable.}

\if{
\begin{remark}
For the interest of the reader, we explicit the expression of the matrices involved in \eqref{lineareq}. For each $t\in [0,T],$ $F_{x}$ is an $n\times n-$matrix given by $\frac{\pr f_0}{\pr x}(\xh,\uh)+\sum_{i=1}^m \vh_{i} \frac{\pr f_i}{\pr x}(\xh,\uh),$ $F_{u}$ is $n\times l$ and is equal to $\frac{\pr f_0}{\pr u}(\xh,\uh)+\sum_{i=1}^m \vh_{i} \frac{\pr f_i}{\pr u}(\xh,\uh)$ and, finally, $F_{v}$ is an $n\times m-$matrix whose $i$th column is $f_i(\xh,\uh),$ for $i=1,\dots,m.$
\end{remark}
}\fi


\subsection{Critical cones}\label{ParCritical}
We define here the sets of critical directions associated with $\wh\cd,$ both in the $L^{\infty}$- and the $L^2$-norms. Even if we are working with control variables in $L^{\infty}$ and hence the control perturbations are naturally taken in $L^{\infty},$ the second order analysis involves quadratic mappings that require to continuously extend the cones to $L^2.$ 

Set
$\X_2:=W^{1,2}(0,T;\cR^n),$ $\U_2:=L^2(0,T;\cR^l)$ and
$\V_2:=L^2(0,T;\cR^m),$ and write $\W_2:=\X_2\times
\U_2\times \V_2$ to refer to the corresponding product
space. 
Given $\wb\cd\in\W_2$ satisfying the linearized state equation \eqref{lineareq}-\eqref{lineareq0}, consider
the \textit{linearization of the endpoint constraints and cost function,}
\begin{gather}
\label{linearconseq}
D\eta_j(\xh(0),\xh(T))(\xb(0),\xb(T))=0,\quad {\rm for}\  j=1,\hdots,d_{\eta},
\\
\label{linearconsineq}
 D\varphi_i(\xh(0),\xh(T))(\xb(0),\xb(T))\leq 0,\quad {\rm for}\ i=0,\hdots,d_{\varphi}.
\end{gather}
The \textit{critical cones} in $\W_2$ and $\W$ are given, respectively, by
\begin{gather}
\label{C2}\C_2:=\{\wb\cd\in\W_2:\text{\eqref{lineareq}-\eqref{lineareq0}}\ \text{and}\ \text{\eqref{linearconseq}-\eqref{linearconsineq}}\ \text{hold}\},\\
\label{C} \C:= \C_2 \cap \W.
\end{gather}
The following density result holds.

\begin{lemma}
\label{conedense}
The critical cone $\C$ is dense in $\C_2$ with respect to the $\W_2$-topology.
\end{lemma}

The proof of previous lemma follows from the
following technical result (due to Dmitruk
\cite[Lemma 1]{Dmi08}).

\begin{lemma}[on density of cones]
\label{lemmadense}
Consider a locally convex topological space $X,$ a
finite-faced cone $Z\subset X,$ and a linear
space $Y$ dense in $X.$ Then the cone $Z\cap Y$ is dense in $Z.$
\end{lemma}


\if{
Let the cone $C$ be given by

\be
C=\{x\in X:(p_i,x)=0,\ \mathrm{for}\ i=
1,\hdots,\mu,\
(q_j,x)\leq 0,\ \mathrm{for}\ j=1,\hdots,\nu\}.
\ee

Let us show first, without lost of generality, that
the
equality constraints can be removed from the
formulation.
It suffices to consider the case where $C$ is given
by only
one equality $(p,x)=0.$

Take any point $x(0)\in C$ and a convex
neighborhood$\mathcal{O}(x(0)).$ We have to show that there
exists $x$
in $C\cap L\cap \mathcal{O}(x(0)).$ Since the set
$(p,x)<0$
is open, its intersection with $\mathcal{O}(x(0))$ is
open
too and obviously nonempty, hence it contains a
point $x_1$
from the set $L,$ because the last one is dense in
$X.$
Similarly, the intersection of the sets $(p,x)<0$
and
$\mathcal{O}(x(0))$ contains a point $x_2\in L.$
Since
$\mathcal{O}(x(0))$ is convex, it contains a point
$x$ such
that $(p,x)=0,$ which belongs to $C$ and to $L\cap
\mathcal{O}(x(0)).$

We can then consider only the case where $C$ is
given by a
finite number of inequalities. Suppose first that
there
exists $\xh\in C$ such that $(q_j,\xh)<0$ for all
$j,$
hence $\xh \in \mathrm{int}{C}.$ Take any $x(0)\in C$
and
any convex neighborhood $\mathcal{O}(x(0)).$ We have
to find
a point $x\in C\cap \mathcal{O}(x(0))\cap L.$ We know
that,
for any positive $\varepsilon,$ the point
$x_{\varepsilon}:=x(0)+\varepsilon \xh$ lies in
$\mathrm{int}(C),$ and then there exists a positive
$\varepsilon$ such that this point lies also in
$\mathcal{O}(x(0)).$ Thus, the open set
$\mathrm{int}{C}\cap
\mathcal{O}(x(0))$ is nonempty, and then contains a
point
$x$ from the dense set $L.$

Suppose now that the above point
$\xh\in\mathrm{int}{C}$
does not exist, and, without lost of generality,
that
$q_j\neq 0$ for every $j.$ In this case, by the
Dubovitskii-Milyutin theorem, there exist
multipliers
$\alpha_j\geq0,\ j=1,\hdots,\nu,$ not all zero, such
that
Euler-Lagrange equation holds: $\alpha_1
q_1+\hdots+\alpha_{\nu}q_{\nu}=0.$ Suppose, without
lost of
generality, that $\alpha_{\nu}>0.$ Then, for all $x
\in C$
we actually have $(q_{\nu},x)=0,$ not just $\leq0.$
This
means that the cone $C$ can be given by the
constraints
$(q_j,x)\leq 0,\ j=1,\hdots,\nu-1,\ (q_{\nu},x)=0.$
But, as
was already shown, the last equality can be removed,
so the
cone can be given by a smaller number of
inequalities.
Applying induction arguments, we arrive at a
situation when
either all the inequalities are changed into
equalities and
then removed, or the strict inequalities have a
nonempty
intersection. Since both cases are already
considered, the
proof is complete.
}\fi


\noindent{\em Proof of Lemma \ref{conedense}.}
Set $X:=\{\wb\cd\in\W_2:\text{\eqref{lineareq}-\eqref{lineareq0}}\, \text{hold}\},$ 
$Y:=\{\wb\cd\in\W:\text{\eqref{lineareq}-\eqref{lineareq0}}\, \text{hold}\},$ and $Z:=\C_2$ and apply Lemma \ref{lemmadense}. The desired density follows.
\findem


\section{Second order analysis}\label{SectionSOC}
We begin this section by giving an expression of the {\em second order derivative of the Lagrangian function $\mathcal{L},$} in terms of derivatives of $\ell$ and $H.$ We let $\Omega$ denote this second variation.
All the second order conditions we present are established in terms of either $\Omega$ or some transformed form of $\Omega.$ The main result of the current section is the necessary condition in Theorem \ref{strengthNC}, which is applied in Section \ref{SectionNC} to get the stronger condition given in  Theorem \ref{NCP2}.

\subsection{Second variation}

Let us consider the quadratic mapping
\be 
\label{Omega}
\begin{split}
\Omega  \lam  &(\xb,\ub,\vb):= \,
\half D^2\ell\lam(\xh(0),\xh(T))(\xb(0),\xb(T))^2  + \intT \Big(\half\xb\tras H_{xx}\lam \xb \ \\
& + \ub\tras H_{ux}\lam\xb +
\vb\tras H_{vx}\lam\xb + \half\ub\tras H_{uu}\lam\ub + \vb\tras
H_{vu}\lam\ub\Big) \dtt.
\end{split}
\ee

\if{
\begin{remark}
Note that $H_{xx}\lam = p\left(\frac{\pr^2 f_0}{\pr x^2} + \sum_{k=1}^m \vh_k \frac{\pr^2 f_k}{\pr x^2}\right)$ is an $n\times n$-matrix, 
$H_{ux}\lam$ is equal to $p\left(\frac{\pr^2 f_0}{\pr u\pr x}+\sum_{k=1}^m \vh_k \frac{\pr^2 f_k}{\pr u\pr x}\right)$ and it is $l\times n,$ 
$H_{vx}\lam$ is an $m\times n-$matrix whose $i$th row is given by $H_{v_ix}\lam= p \frac{\pr f_i}{\pr x}.$ The matrix $H_{uu}\lam$ is equal to $p\left(\frac{\pr^2 f_0}{\pr u^2}+\sum_{k=1}^m \vh_k \frac{\pr^2 f_k}{\pr u^2}\right)$ and it is $l\times l,$ and the $i-$th. row of the $m\times l-$matrix $H_{vu}\lam$ is $H_{v_i u}\lam= p \frac{\pr f_i}{\pr u}.$
\end{remark}
}\fi


The result that follows gives an expression of the Lagrangian $\mathcal{L}$ at the nominal trajectory $\wh\cd.$ For the sake of simplicity, the time variable is omitted in the statement.

\begin{lemma}[Lagrangian expansion] 
\label{expansionlagrangian}
Let $w\cd=(x,u,v)\cd\in\W$ be a trajectory and set $\delta w\cd=(\delta x,\delta u,\delta v)\cd:=w\cd-\wh\cd.$
Then, for every multiplier $\lambda\in\Lambda,$ the following expansion of the Lagrangian holds
\be
\label{expansionLag}
\mathcal{L}\lam(w) 
=\mathcal{L}\lam(\wh) 
+ \Omega\lam(\delta x,\delta u,\delta v)  
+\omega\lam(\delta x,\delta u,\delta v)  +\R(\delta x,\delta u,\delta v),
\ee
where  $\omega$ is a cubic mapping given by
\begin{align*}
\omega\lam&(\delta x,\delta u,\delta v)  := \\
&\intT \left[ H_{vxx}\lam(\delta x, \delta x,\delta v) + 2H_{vux}\lam(\delta x,\delta u,\delta v) 
+ H_{vuu}\lam(\delta u,\delta u,\delta v) \right] \dtt,
\end{align*}
and $\R$ satisfies the estimate
\benl
\R(\delta x,\delta u,\delta v)  = L_\ell |(\delta x(0),\delta x(T))|^3
+L K (1+\|v\|_\infty)\, \|(\delta x, \delta u)\|_{\infty}\|(\delta x,\delta u) \|_2^2.
\eenl
Here $L_\ell$ is a Lipschitz constant for $D^2\ell\lam$ uniformly with respect to $\lambda\in \Lambda,$ $L$ is a Lipschitz constant for $D^2f_i$ uniformly in $i=0,\dots,m,$ and $K:=\ds\sup_{\lambda\in \Lambda} \|p\cd\|_\infty.$
\end{lemma}

\begin{proof}
 See Appendix \ref{proofexpansionlagrangian}.
\end{proof}

\begin{remark} From previous lemma one gets the identity
\benl 
\Omega\lam (\wb) = \half D^2\mathcal{L}\lam (\wh)\, \wb^2.
\eenl
\end{remark}


\subsection{Second order necessary condition}

The following result is a classical second order condition for weak minima.
\begin{theorem}[Second order necessary
condition]
\label{classicalNC}
 If $\wh\cd$ is a weak minimum of problem (P), then
\be \label{classicalNCeq}
\max_{\lambda\in \Lambda} \Omega \lam (\xb,\ub,\vb)
\geq 0,\ \mr{for}\ \mr{all}\ (\xb,\ub,\vb)\in\C.
\ee
\end{theorem}

A proof of Theorem \ref{classicalNC} can be found in Levitin, Milyutin and Osmolovskii \cite{LevMilOsm1985}.
Nevertheless, for the sake of completeness, we give a proof in the Appendix \ref{AppendixclassicalNC}  that uses techniques of optimization in abstract spaces.

An extension of the condition \eqref{classicalNCeq} to the cone $\C_2$ can be easily proved and gives the following, stronger, second order condition.

\begin{theorem}
\label{NCC2}
 If $\wh\cd$ is a weak minimum of problem (P), then
\be 
\label{classicalNCeqC2}
\max_{\lambda\in \Lambda} \Omega \lam (\xb,\ub,\vb)
\geq 0,\quad \mr{for}\ \mr{all}\ (\xb,\ub,\vb)\in \C_2.
\ee
\end{theorem}

\begin{proof}
Observe first that $\Omega\lam$ can be extended to the space $\W_2$ since all the coefficients are essentially bounded. The result follows by  the density property of Lemma \ref{conedense} and  the compactness of the Lagrange multipliers set $\Lambda$ proved in Theorem \ref{LambdaCompact}. 
\end{proof}


\subsection{Strengthened second order necessary condition}

In the sequel we aim at strengthening the necessary condition of Theorem \ref{NCC2} by proving that the maximum in \eqref{classicalNCeqC2} remains nonnegative when taken in a {\em possibly} smaller set of multipliers, whenever $\Lambda$ is convex.

Let ${\rm co}\, \Lambda$  denote  the {\em convex hull} of $\Lambda.$ Observe that if $\lambda=(\alpha,\beta,p(\cdot))$ is in $  {\rm co}\, \Lambda$ then it verifies \eqref{alphapos}-\eqref{stationarity} and, if $\wh(\cdot)$ is a weak minimum,  also the second order condition \eqref{classicalNCeqC2} is fulfilled for $\lambda.$
However, $\lambda$ may not verify the nontriviality condition \eqref{nontriv}, thus ${\rm co}\, \Lambda$ may content the trivial (i.e. identically zero) multiplier.

Set
\be\label{H2}
\H_2:=\{(\xb,\ub,\vb)\cd\in\W_2:\eqref{lineareq}\ \text{holds}\},
\ee
and consider the subset of ${\rm co}\, \Lambda$ given by
\benl 
({\rm co}\, \Lambda)^{\#}:=\{ \lambda\in {\rm co}\, \Lambda: \Omega\lam\
\text{is weakly-l.s.c. on }\H_2\}.
\eenl
Next we prove that $({\rm co}\, \Lambda)^{\#}$ can be characterized in a quite simple way (see Lemma \ref{Lambdawlsc} below). Theorem \ref{strengthNC} stated afterwards yields a {\em new}  necessary optimality condition. 

\begin{lemma}
\label{Lambdawlsc}
\be 
({\rm co}\, \Lambda)^{\#} =\{ \lambda\in {\rm co}\, \Lambda: H_{uu}
\lam\succeq0\ {\rm and}\ H_{vu}\lam = 0,\,\, {\rm a.e.}\, {\rm on}\,[0,T]\}.
\ee
\end{lemma}

\begin{remark}[About {\em singular} solutions]
\label{RemSing}
From now on we restrict the set $({\rm co}\, \Lambda)^{\#}$ or some subset of it and, therefore, $H_{uv}\lam\equiv 0$ along the nominal trajectory $\wh\cd.$ Consequently, 
\benl
D^2_{(u,v)^2}H\lam (\wh,t)\ \mr{is}\ \mr{a}\ \mr{singular}\ \mr{matrix}\ \mr{a.e.\,\, on}\  [0,T].
\eenl
The latter assertion together with the stationarity condition \eqref{stationarity} imply that $(\wh,\lambda)$ is a {\em singular extremal} (as defined in Bryson-Ho \cite[Page 246]{BryHo}). That is, if we write $\nu:=(u,v)$ for the control, we say that $(\wh, \lambda)$ is a singular extremal if $H_\nu[\lambda]=0$ and $H_{\nu\nu}[\lambda]$ is singular a.e. on $[0,T]$.

Let us comment on the terminology used in the literature for the class of problems where $H_{\nu\nu}$ is a singular matrix.
In Bell-Jacobson \cite[Definition 1.2]{BelJac}  and Ruxton-Bell \cite{RuxtonBell1995} they refer to singular extremals (as defined above) as {\em totally singular}, while they use the term {\em partially singular} to refer to controls for which $H_{\nu}=0$ only on some subintervals of $[0,T],$ which is not the class of controls studied here. The same definition is adopted in Poggiolini and Stefani \cite{PogioliniStefani2007}.  On the other hand, O'Malley in \cite{OMalley1977} calls {\em partially singular} the linear-quadratic problems in which the matrix $H_{\nu\nu}$ is (singular but) not of constant non-zero rank, that is a framework included in our class of problems.
\end{remark}

In order to prove  Lemma \ref{Lambdawlsc} we shall notice that $\Omega\lam$ can be written as the sum of two maps: the first one being a weakly-continuous function on the space $\H_2$ given by
\be\label{Omega1}
(\xb,\ub,\vb)\mapsto\half D^2\ell\lam(\xb(0),\xb(T))^2  + \intT \Big( \half\xb\tras H_{xx}\lam \xb + \ub\tras H_{ux}\lam\xb +
\vb\tras H_{vx}\lam\xb \Big) \dtt,
\ee
 and the second one being the quadratic operator
\be\label{Omega2}
(\ub,\vb)\mapsto 
\intT \Big( \half\ub\tras H_{uu} \lam\ub + \vb\tras
H_{vu}\lam\ub \Big) \dtt.
\ee
The weak-continuity of the mapping in \eqref{Omega1} follows easily. Additionally, in view of Hestenes \cite[Theorem 3.2]{Hes51}, the following characterization holds.

\begin{lemma}
\label{Lemmawlsc}
The mapping in \eqref{Omega2} is weakly-lower semicontinuous on $\U\times \V$ if and only if the matrix
\be 
\label{Huvuv}
D^2_{(u,v)^2}H\lam=\begin{pmatrix}
H_{uu}\lam & H_{vu}\lam\tras \\
 H_{vu}\lam & 0\\
\end{pmatrix},
\ee
is positive semidefinite almost everywhere on $[0,T].$
\end{lemma}

\begin{remark}
\label{remarkLC}
The fact that the matrix in \eqref{Huvuv} is positive semidefinite is known as the {\it Legendre-Clebsch necessary optimality condition} for the extremal $(\wh,\lambda)$ (see e.g. Bliss \cite{Bliss1946} in the framework of Calculus of Variations, and Bryson-Ho \cite{BryHo}, Agrachev-Sachkov \cite{AgrSac} or Corollary \ref{NCunique} below for Optimal Control). 
\end{remark}

We can now prove Lemma \ref{Lambdawlsc}.

\noindent{\em Proof of Lemma \ref{Lambdawlsc}.}
It follows from the decomposition given in \eqref{Omega1}-\eqref{Omega2} and the characterization of weak-lower semicontinuity stated in previous Lemma \ref{Lemmawlsc}.
\findem

\begin{theorem}[Strengthened second order necessary condition]
\label{strengthNC}
 If $\wh\cd$ is a weak minimum of problem (P), then
\be \label{strengthNCeq}
\max_{\lambda\in ({\rm co}\, \Lambda)^{\#}} \Omega \lam
(\xb,\ub,\vb) \geq 0,\quad \mr{on}\ \C_2.
\ee
\end{theorem}

\begin{remark}[On {\em unqualified} solutions]
Notice that it may occur that $0\in ({\rm co}\, \Lambda)^{\#}$ and, in this case, the second order condition in Theorem \ref{strengthNC} above does not provide any information. 
This situation may arise when the endpoint constraints are {\it not qualified,} in the sense of the {\em constraint qualification condition \eqref{QC}} introduced in the Appendix, which is a natural generalization of the {\em Mangasarian-Fromovitz} condition \cite{ManFro67} to the infinite-dimensional framework.
\end{remark}

In order to achieve Theorem \ref{strengthNC}, let us recall the following result on quadratic forms (taken from Dmitruk \cite[Theorem 5]{Dmi84}).

\begin{lemma}
\label{quadform}
Given a Hilbert space $H,$ and
$a_1,a_2,\hdots,a_p$ in $H,$ set
\be
K:=\{x\in H:(a_i,x)\leq 0,\ \mr{for}\
i=1,\hdots,p\}.
\ee
Let $M$ be a convex and compact subset of $\cR^s,$
and let
$\{Q^{\psi}:\psi\in M\}$ be a family of continuous
quadratic forms over $H,$ the mapping $\psi
\rightarrow
Q^{\psi}$ being affine. Set 
$M^{\#}:=\{ \psi \in M:\ Q^{\psi}\ \text{is
weakly-l.s.c.}\text{ on } H\}$ and assume that
\be
\max_{\psi\in M} Q^{\psi}(x)\geq 0,\ \mr{for}\
\mr{all}\ x\in K.
\ee
Then
\be
\max_{\psi\in
M^{\#}} Q^{\psi}(x)\geq 0,\ \mr{for}\ \mr{all}\ x\in K.
\ee
\end{lemma}

We are now able to show Theorem \ref{strengthNC} as desired.

\noindent{\em Proof of Theorem \ref{strengthNC}.}
It is a consequence of Theorem \ref{NCC2}, Lemmas \ref{Lambdawlsc} and \ref{quadform}. 

\findem

We finish this section with the following extension of the classical second order pointwise Legendre-Clebsch condition, which follows as a corollary of Theorem \ref{strengthNC}.

\begin{corollary}[Legendre-Clebsch condition]
\label{NCunique}
If $\wh\cd$ is a weak minimum of (P) with a unique associated Lagrange multiplier $\hat\lambda,$ then $(\wh,\hat\lambda)$ satisfies the {\em Legendre-Clebsch condition}, this is, the matrix in \eqref{Huvuv} is positive semidefinite and, consequently,
\be
\label{R0pos}
 H_{uu} \lamh\succeq0\ {\rm and}\ H_{vu}\lamh \equiv 0.
\ee
\end{corollary} 

\begin{proof}
It follows easily from Theorem \ref{strengthNC}. In fact, as the Lagrange multiplier is unique, ${\rm co}\, \Lambda = \Lambda= \{\hat\lambda\},$ and the inequality in \eqref{strengthNCeq} implies that $({\rm co}\, \Lambda)^{\#}\neq \emptyset.$  Therefore, $({\rm co}\, \Lambda)^{\#}=\Lambda^{\#} = \{\hat\lambda\}$ and \eqref{R0pos} necessarily holds.
\end{proof}


\section{Goh Transformation}\label{GohT}

In this section we introduce the {\em Goh trasformation} which is a linear change of variables applied usually to a linear differential equation, and that is motivated by the facts explained in the sequel. In the previous section we were able to provide a necessary condition involving the nonnegativity on $\C_2$ of the maximum of $\Omega\lam$ over the set $({\rm co}\, \Lambda)^{\#}$ (Theorem \ref{strengthNC}). Our next step  is finding a sufficient condition. To achieve this one would naturally try to strengthen the inequality \eqref{strengthNCeq} to convert it into a condition of strong positivity. However,  since no quadratic term on $\vb\cd$ appears in $\Omega,$ the latter cannot be strongly positive with respect to the norm of the controls. Thus, what we do here to find the desired sufficient condition is transforming $\Omega$ into a new quadratic mapping that may result strongly positive on an appropriate transformed critical cone. For historical interest, we recall that Goh introduced this change of variables in \cite{Goh66a} and employed it to derive necessary conditions in \cite{Goh66a,Goh66}. Since then, many optimality conditions were obtained by using that transformation as already mentioned in the Introduction. 

\vspace{4pt}

For the remainder of the article, we consider the following regularity hypothesis on the controls.
\begin{assumption}
\label{SmoothControls}
The controls $\uh\cd$ and $\vh\cd$ are smooth. 
\end{assumption}
This hypothesis is not restrictive since it is a consequence of the {\em strengthened generalized Legendre-Clebsch condition} as explained in Aronna \cite{Aro11,Aro13}, where it is shown that, whenever this generalized condition holds, one can write the controls as smooth functions of the state and costate variable. See also Remark \ref{RemarkLC} below.

\vspace{4pt}

Consider hence the linearized state equation \eqref{lineareq} and
the {\em Goh transformation} defined by 
\be \label{Goht}
\left\{
\ba{l}
\yb(t):= \ds\int_0^t \vb(s) {\rm d}s, \\
\xib(t) := \xb(t)-  F_{v}[t]\,\yb(t), 
\ea
\right.
\quad {\rm for}\ t\in [0,T].
\ee
Observe that $\xib\cd$ defined in that way satisfies the linear equation
\be \label{xieq}
\dot\xib  = F_{x}\,\xib + F_{u}\,\ub +B\,\yb,\quad
\xib(0)=\xb(0),
\ee
where 
\be
\label{B1}
B:= F_{x} F_{v}-\ddt F_{v}.
\ee
Here $B$ is an $n\times m$-matrix whose  $i$th column is given by
\benl
-[f_i,f_0]^x-\sum_{j=1}^m \vh_j[f_i,f_j]^x + D_u f_i\, \dot{\uh},
\eenl
where $[f_i,f_j]^x:=({\rm D}_xf_i)f_j-(D_xf_j)f_i$ and it is referred as the {\it Lie bracket with respect to $x$} of the vector fields $f_i$ and $f_j.$


\subsection{Tranformed critical cones}
In this paragraph we present the critical cones
obtained after Goh's transformation.
We shall recall the linearized endpoint constraints \eqref{linearconseq}-\eqref{linearconsineq} and the critical cones \eqref{C2}-\eqref{C}.
Let $(\xb,\ub,\vb)\cd\in \C$ be a critical
direction. Define $(\xib,\yb)\cd$ by Goh's transformation
\eqref{Goht} and set $\hb:=\yb(T).$
From \eqref{linearconseq}-\eqref{linearconsineq} we get
\begin{gather}
\label{tlinearconseq} D\eta_j (\xh(0),\xh(T))\big(\xib(0),\xib(T)+F_{v}[T]\hb\big)=0,\quad
\mr{for}\,\, j=1,\hdots,d_{\eta},
\\
\label{tlinearconsineq}
D\varphi_i(\xh(0),\xh(T))\big(\xib(0),\xib(T)+F_{v}[T]\hb\big)\leq
0,\quad \mr{for}\,\,
i=0,\hdots,d_{\varphi}.
\end{gather}
Remind the definition of the linear space $\W_2$ given in paragraph \ref{ParCritical}.
Let $\Y$ denote the Sobolev space $W^{1,\infty}(0,T;\cR^m),$
and consider the cones
\be 
\label{P}
\P:= \{(\xib\cd, \ub\cd,\yb\cd,\hb)\in \W \times
\cR^m:\,\yb(0)=0,\,\yb(T)=\hb,\,\text{\eqref{xieq},
\eqref{tlinearconseq}-\eqref{tlinearconsineq} hold} \},
\ee
\be \label{P2}
\P_2:= \{(\xib\cd, \ub\cd,\yb\cd,\hb)\in \W_2 \times
\cR^m:\,\text{\eqref{xieq}, \eqref{tlinearconseq}-\eqref{tlinearconsineq}
hold} \}.
\ee
\begin{remark}
 \label{PandC}
Observe that $\P$ is the cone obtained from $\C$ via Goh's transformation \eqref{Goht}.
\end{remark}
The next result shows the density of $\P$ in $\P_2.$ This fact is used afterwards when we extend a necessary condition stated in $\P$ to the bigger cone $\P_2$ by continuity arguments, as it was done for $\C$ and $\C_2$ in Section \ref{SectionSOC}.
\begin{lemma} \label{PdenseP2}
$\P$ is a dense subspace of $\P_2$ in the $ \W_2
\times \cR^m$-topology.
\end{lemma}

\begin{proof}
Notice that the inclusion $\P\subset \P_2$ is immediate. In order to prove the density, consider the linear spaces
\begin{gather*}
 X:= \{ (\xib\cd,\ub\cd,\yb\cd,\hb)\in \W_2\times \cR^m:\ \eqref{xieq}\ {\rm holds}\},\\
Y:=\{ (\xib\cd,\ub\cd,\yb\cd,\hb)\in \W\times \cR^m:\ \yb(0)=0,\, \yb(T)=\hb\ {\rm and}\ \eqref{xieq}\ {\rm holds}\},
\end{gather*}
and the cone
\benl
Z:= \{ (\xib\cd,\ub\cd,\yb\cd,\hb)\in X:\ \text{\eqref{tlinearconseq}-\eqref{tlinearconsineq}}\ {\rm holds}\}.
\eenl
Notice that $Y$ is a dense linear subspace of $X$  (Dmitruk-Shishov  \cite[Lemma 6]{DmiShi10} or Aronna et al. \cite[Lemma 8.1]{ABDL11}), and $Z$
is a finite-faced cone of $X. $ The desired density follows by Lemma \ref{lemmadense}. 
\end{proof}


\subsection{Transformed second variation}
Next we write the quadratic mapping $\Omega$ in the variables $(\xib\cd,\ub\cd,\yb\cd,\vb\cd,\hb).$
Set, for $\lambda\in ({\rm co}\, \Lambda)^{\#},$
\be 
\label{OmegaP}
\begin{split}
 \Omega_{\P}  &\lam (\xib,\ub,\yb,\vb,\hb) 
:= g\lam (\xib(0),\xib(T),\hb)
+ \ds \intT \left( \half\xib\,\tras H_{xx}\lam \xib + \ub\tras
H_{ux}\lam\xib \right.\\
 &\left. +\, \yb\tras M\lam \xib + \half\ub\tras H_{uu}\lam
\ub + \yb\tras E\lam \ub +
\half\yb\tras R\lam \yb + \vb\tras G\lam \yb
  \right) \dtt,
\end{split}
\ee 
where 
\begin{gather}
\label{M}
M:= F_v\tras H_{xx}-\dot H_{vx}-H_{vx}F_x,\quad E:= F_v\tras H_{ux}\tras - H_{vx}F_u,
\\
\label{SV}
S:=\half (H_{vx}F_v + (H_{vx}F_v)\tras),\quad
G:= \half (H_{vx}F_v - (H_{vx}F_v)\tras),
\\
\label{R1}
R := F_v\tras H_{xx}F_v - (H_{vx}B+(H_{vx}B)\tras) -
\dot S,
\\
\label{g}
g\lam (\xi_0,\xi_T,h):=
\half\ell''(\xi_0,\xi_T+F_{v}[T]\,h)^2
+h\tras(H_{vx}[T]\, \xi_T+\half S[T] h).
\end{gather}
Observe that, in view of Assumptions \ref{regular} and \ref{SmoothControls}, all the functions defined above are continuous in time. 

\begin{remark}
We can see that $M$ is an $m\times n$-matrix whose $i$th row is given by the formula 
\benl
M_i = p\sum_{j=0}^m \vh_j  \left(  \frac{\pr^2 f_j}{\pr x^2} f_i - 
 \frac{\pr^2 f_i}{\pr x^2} f_j +  \frac{\pr f_j}{\pr x} \frac{\pr f_i}{\pr x} - \frac{\pr f_i}{\pr x} \frac{\pr f_j}{\pr x} \right)- p \frac{\pr^2 f_i}{\pr x\pr u} \dot{\uh},
\eenl
$E$ is $m\times l$ with
$
E_{ij}=p\ds \frac{\pr^2 F}{\pr u_j \pr x} f_i - p  \frac{\pr f_i}{\pr x}\frac{\pr F}{\pr u_j},
$
the $m\times m-$matrices $S$ and $G$ have entries
$
S_{ij} = \ds\half p \left( \frac{\pr f_i}{\pr x}f_j + \frac{\pr f_j}{\pr x}f_i \right),
$
and
\be
\label{Vcrochet}
G_{ij}=p[f_i,f_j]^x,
\ee
respectively.
The components of the matrix $R$ have a quite long expression, that is simplified for some multipliers as it is detailed in equation  \eqref{Rij} in the next section.
\end{remark}

The identity between $\Omega$ and $\Omega_\P$ stated in the following lemma holds.
\begin{lemma} 
\label{Omegat}
Let $\lambda \in ({\rm co}\, \Lambda)^{\#},$ $(\xb,\ub,\vb)\cd \in \H_2$ (given in \eqref{H2}) and $(\xib,\yb)\cd$ be defined by Goh's transformation
\eqref{Goht}. Then
\benl
\Omega\lam (\xb,\ub,\vb) = \Omega_{\P} \lam
(\xib,\ub,\yb,\vb,\yb(T)).
\eenl
\end{lemma}
The proof of this lemma is merely technical and we leave it to the Appendix \ref{AppendixOmegat}.

Finally let us remind the strengthened necessary condition of Theorem \ref{strengthNC}. Observe that by  Goh's transformation \eqref{strengthNCeq} and in view of Remark \ref{PandC}, we obtain the following form of the second order necessary condition.
\begin{corollary}
 \label{transNC}
 If $\wh\cd$ is a weak minimum of problem (P), then
\be \label{maxOmegaP}
\max_{\lambda\in ({\rm co}\, \Lambda)^{\#}} \Omega_{\P} \lam
(\xib,\ub,\yb,\dot\yb,\hb) \geq 0,\quad \mr{on}\ \P.
\ee
\end{corollary}


\section{New second order necessary condition}\label{SectionNC}

We aim at removing the dependence on $\vb$ in the formulation of the  second order necessary condition of Corollary \ref{transNC} above.  
Note that in the inequality \eqref{maxOmegaP}, $\vb=\dot \yb$ appears only in the term $\vb\tras G\lam \yb.$  We prove in the sequel that we can
restrict the maximum in \eqref{maxOmegaP} to the
subset of $({\rm co}\, \Lambda)^{\#}$ consisting of the
multipliers for which $G\lam$ vanishes.

Let $G({\rm co}\, \Lambda)^{\#}$ refer to  the subset of $({\rm co}\, \Lambda)^{\#}$ for which $G\lam$ vanishes, i.e. 
\be
G({\rm co}\, \Lambda)^{\#}:=\{\lambda\in ({\rm co}\, \Lambda)^{\#}: G\lam \equiv 0\}.
\ee
Hence, the following optimality condition holds.
\begin{theorem}[New necessary condition]\label{newNC}
 If $\wh\cd$ is a weak minimum of problem (P), then
\be
\max_{\lambda\in G({\rm co}\, \Lambda)^{\#}} 
\Omega_{\P} \lam (\xib,\ub,\yb,\dot\yb,\yb(T)) \geq
0,\quad \text{on}\ \P.
\ee
\end{theorem}

Theorem \ref{newNC} is an extension of similar results given in Dmitruk \cite{Dmi77},  Milyutin \cite{Mil81} and recently in Aronna et al. \cite{ABDL11}. The proof given in Aronna et al. \cite[Theorem 4.6]{ABDL11} holds for Theorem \ref{newNC} with minor modifications and hence we do not include it in the present article.

Notice that when $\wh\cd$ has a unique associated multiplier, from Theorem \ref{newNC} one can deduce that $G({\rm co}\, \Lambda)^{\#}$ is not empty, and since the latter is a singleton, the corollary below follows. This result gives an extension of the necessary conditions stated by Goh in \cite{Goh66} to the present framework.

\begin{corollary}[Goh conditions] 
\label{CoroCBsym}
 Assume that $\wh\cd$ is a weak minimum having a unique associated multiplier. Then the following conditions holds. \begin{itemize} \item[(i)] $G\equiv 0 $ or, equivalently, the matrix $H_{vx}F_v$ is symmetric, which, in view of \eqref{Vcrochet}, can be written as
\benl
p[f_i,f_j]^x\cd\equiv0,\quad \text{for}\ i,j=1,\dots,m,
\eenl
where $p\cd$ is the unique associated adjoint state.
\item[(ii)] The matrix
\be
\label{R2}
\begin{pmatrix}
H_{uu} & E^\top \\
E& R
\end{pmatrix}
\ee
is positive semidefinite.
\end{itemize}
\end{corollary}

We aim now at stating a necessary condition that does not depend on $\vb\cd.$
Let us note that, for $\lambda\in
G({\rm co}\, \Lambda)^{\#},$ the quadratic form
$\Omega\lam$ does not depend on $\vb\cd$ since
its coefficients vanish. We can then consider its
continuous extension to $\P_2$ for multipliers $\lambda\in G({\rm co}\, \Lambda)^{\#},$  given by
\be 
\label{OmegaP2}
\begin{split}
\Omega_{\P_2}\lam(\xib,\ub,\yb,\hb):= &\,g\lam
(\xib(0),\xib(T),\hb)
+ \ds \intT \left( \half\xib\,\tras H_{xx}\lam \xib + \ub\tras
H_{ux}\lam\xib \right.\\
 &\left. +\, \yb\tras M\lam \xib + \half\ub\tras H_{uu}\lam
\ub + \yb\tras E\lam \ub + \half\yb\tras R\lam \yb
\right) \dtt,
\end{split}
\ee
where the involved matrices and the function $g$ were defined in  \eqref{M}-\eqref{g}. Observe that, since $G\lam \equiv 0,$ one has that $H_{vx}\lam F_v$ is symmetric and, therefore, the $ij$ entry of $R\lam$ can be written as
\be
\label{Rij}
\begin{split}
R_{ij}\lam=&-p\left\{ [f_j,[f_0,f_i]^x]^x + \sum_{k=1}^m \vh_k [f_j,[f_k,f_i]^x]^x \right.\\
&+ \left.\left(  2\frac{\pr f_i}{\pr x} \frac{\pr f_j}{\pr u} +\frac{\pr f_j}{\pr x} \frac{\pr f_i}{\pr u} +\frac{\pr^2 f_i}{\pr u\pr x} f_j \right) \dot{\uh} \right\},
\end{split}
\ee
for each $i,j=1,\dots,m.$

From Theorem \ref{newNC}, it follows:
\begin{theorem}[Second order necessary condition in new variables] 
\label{NCP2}
 If $\wh\cd$ is a weak minimum of problem (P), then
\be
\label{Omegapos}
\max_{\lambda\in G({\rm co}\, \Lambda)^{\#}} 
\Omega_{\P_2} \lam (\xib,\ub,\yb,\hb) \geq 0,\quad
\text{on}\ \P_2.
\ee
\end{theorem}



\section[Second order sufficient condition]{Second order sufficient condition  for weak minimum}\label{SectionSC}

In this section we present the main contribution of the article: a second order sufficient condition for strict weak optimality.
The optimality to be investigated here is with respect to the following   \textit{$\gamma$-order:} 
\be 
\gamma_\P\big(\xb(0),\ub\cd,\yb\cd,\hb\big):=|\xb(0)|^2+ |\hb|^2+
\intT(|\ub(t)|^2+|\yb(t)|^2)\dtt,
\ee
defined for $(\xb(0),\ub\cd,\yb\cd,\hb)\in \cR^n\times\U_2\times \V_2 \times
\cR^{m}.$
Let us note that $\gamma_\P$ can also be considered as a function of
$(\xb(0),\ub\cd,\vb\cd)\in \cR^n\times \U_2\times \V_2$ by setting
\be 
\gamma(\xb(0),\ub\cd,\vb\cd):=
\gamma_\P(\xb(0),\ub\cd,\yb\cd,\yb(T)),
\ee
with $\yb\cd$ being the primitive of $\vb\cd$ defined as in Goh transform \eqref{Goht}.

This $\gamma$-order was proposed in Dmitruk \cite{Dmi11} for a simpler {\em partially-affine} problem and it is a natural extension of the order suggested (for control-affine problems) in Dmitruk \cite{Dmi77}.

\begin{definition}\label{qgdef}[$\gamma$-growth]
We say that $\wh\cd$ satisfies the
$\gamma$-\textit{growth condition in the
weak sense} if there exist $\eps,\rho\gr 0$ such
that 
\be \label{qg}
\varphi_0(x(0),x(T)) \geq \varphi_0(\xh(0),\xh(T)) + \rho \gamma(x(0)-\xh(0),u\cd-\uh\cd,v\cd-\vh\cd),
\ee 
for every feasible trajectory $w\cd$ with $\|w\cd-\wh\cd\|_{\infty} \mi \eps.$
 
\end{definition}


\begin{theorem}[Sufficient condition for weak optimality]\label{SC}
\begin{itemize}
\item[(i)]
Assume that there exists $\rho \gr 0$ such that
\be
\label{unifpos}
\max_{\lambda\in G(\mr{co}\,\Lambda)^{\#}} 
\Omega_{\P_2} \lam (\xib,\ub,\yb,\hb) \geq
\rho\gamma_\P(\xib(0),\ub,\yb,\hb),\quad \text{on}\
\P_2.
\ee
Then $\wh\cd$ is a weak minimum satisfying 
$\gamma$-growth in the weak sense.
\item[(ii)] Conversely, if $\wh\cd$ is a weak solution satisfying $\gamma$-growth in the weak sense and such that $\alpha_0>0$ for every $\lambda\in G(\mr{co}\,\Lambda)^{\#},$ then \eqref{unifpos} holds for some positive $\rho.$
\end{itemize}
\end{theorem}

In the absence of the {\em nonlinear control} $u,$ Theorem \ref{SC} was proved in Dmitruk \cite{Dmi77}. In Aronna et al. \cite{ABDL11} the same result was shown for the case of scalar control subject to bounds.

As a consequence of Theorem \ref{SC} and standard results on positive quadratic mappings due to Hestenes \cite{Hes51} we get the following pointwise condition. 
\begin{corollary}
\label{CoroSC}
If $\wh\cd$ satisfies the uniform positivity in \eqref{unifpos} and it has a unique associated multiplier, then the matrix in \eqref{R2} is uniformly positive definite, i.e.
$$
\begin{pmatrix}
H_{uu} & E^\top \\
E & R
\end{pmatrix}
\succeq \rho I,\quad \text{on} \ [0,T],
$$
where $I$ refers to the identity matrix.
\end{corollary}

\begin{remark}
\label{RemarkLC}
Under suitable hypotheses, Goh in \cite{GohThesis} proved that the {\it strengthened generalized Legendre-Clebsch condition} is a consequence of the uniform positivity in \eqref{unifpos} (see Goh \cite[Section 4.8]{GohThesis} and Aronna \cite[Remark 8.2]{Aro11}). Thus, in that situation, the controls can be expressed as smooth functions of the state and costate variable, as was assumed here.
\end{remark}

The remainder of this section is devoted to the proof of Theorem \ref{SC}. Several technical lemmas that are used in the following proof were stated and proved in the Appendix \ref{AppendixSC}. 

\vspace{5pt}


\noindent{\em Proof of Theorem \ref{SC}.}
{\em (i)}  We shall prove that if \eqref{unifpos} holds for
some $\rho>0,$ then $\wh\cd$ satisfies
$\gamma$-growth in the weak sense.
By the contrary, let us assume that the
$\gamma$-growth condition \eqref{qg} is not satisfied.
Consequently, there exists
a sequence of feasible trajectories $\{w_k\cd=(x_k\cd,u_k\cd,v_k\cd)\}$ converging to $\wh\cd$ in the weak sense, such that
\begin{equation}
\label{qgrowth}
\varphi_0(x_k(0),x_k(T))\leq \varphi_0(\xh(0),\xh(T))+o(\gamma_k),
\end{equation}
with 
\benl
(\delta x_k\cd,\ub_k\cd,\vb_k\cd):= w_k\cd-\wh\cd \,\,\, \text{and} \,\,\,
\gamma_k:= \gamma (\delta x_{k}(0),\ub_k\cd,\vb_k\cd).
\eenl
Let $(\xib_k\cd,\ub_k\cd,\yb_k\cd)$ be the transformed directions defined by Goh transformation \eqref{Goht}. We divide the remainder of the proof of item {\em (i)} in the following two steps:
 \begin{itemize}
\item[(A)] First  we prove that the sequence given by 
\benl
(\mathring{\xi}_k\cd,\mathring{u}_k\cd,\mathring{y}_k\cd,\mathring{h}_k):= (\xib_k\cd,\ub_k\cd,\yb_k\cd,\hb_k)/{\sqrt{\gamma_k}}
\eenl
 where $\hb_k:=\yb_k(T),$ contains a weak converging subsequence whose weak limit is an element\\ $(\mathring\xi\cd,\mathring u\cd,\mathring y\cd,\mathring h)$ of $\P_2.$ 
\item[(B)] 
Afterwards, making use of the latter sequence and its weak limit, we show that the uniform positivity hypothesis \eqref{unifpos} together with \eqref{qgrowth} lead to a contradiction.
\end{itemize}


We shall begin by {Part (A).} For this we take an arbitrary Lagrange multiplier $\lambda$ in $(\mr{co}\,\Lambda)^{\#}.$
By multiplying the inequality \eqref{qgrowth} by $\alpha_0,$ and adding the nonpositive term
\be
\sum_{i=1}^{d_{\varphi}}\alpha_i\varphi_i(x_{k}(0),x_{k}(T))+\sum_{j=1}^{d_{\eta}}\beta_j\eta_j(x_{k}(0),x_k(T)),
\ee
to its left-hand side, we get
\begin{equation}
\label{quadlag}
\mathcal{L}[\lambda](w_k)\leq\mathcal{L}[\lambda](\wh)+o(\gamma_k).
\end{equation}

Note that the elements of the sequence $(\mathring\xi_{k}(0),\mathring u_k\cd,\mathring y_k\cd,\mathring h_k)$ have unit $\cR^n\times \U_2\times \V_2\times \cR^m$-norm. The Banach-Alaoglu Theorem (see e.g. Br\'ezis  \cite[Theorem III.15]{Bre83})  implies that, extracting if
necessary a subsequence, there exists
$(\mathring\xi(0),\mathring u\cd,\mathring y\cd,\mathring h)\in \cR^n\times \U_2\times \V_2\times \cR^m$ 
such that
\begin{equation}
\label{limityk}
\mathring\xi_{k}(0)\rightarrow \mathring\xi(0),\quad
\mathring u_k\rightharpoonup \mathring u,\quad
\mathring y_k\rightharpoonup \mathring y,\quad
\mathring h_k\rightarrow \mathring h,
\end{equation}
where the two limits indicated with
$\rightharpoonup$ are considered in the weak topology of
$\U_2$ and $\V_2,$ respectively.
Let $\mathring\xi\cd$ denote the solution of the equation \eqref{xieq} associated with
$(\mathring \xi(0),\mathring u\cd,\mathring y\cd).$ Hence, it follows easily that $\mathring\xi\cd$ is the limit of $\mathring\xi_k\cd$ in (the strong topology of) $\X_2.$ 

With the aim of proving that $(\mathring\xi\cd,\mathring u\cd,\mathring v\cd,\mathring h)$ belongs to $\P_2,$ it remains to check that the linearized endpoint constraints \eqref{tlinearconseq}-\eqref{tlinearconsineq} are verified.
Observe that, for each index $0\leq i\leq d_{\varphi},$ one has
\be 
\label{phineg'}
D\varphi_i(\xh(0),\xh(T)) (\mathring\xi(0),\mathring\xi(T)+B[T]\mathring h)
=  
\lim_{k\rightarrow \infty} D\varphi_i(\xh(0),\xh(T))\left(\frac{\xb_{k}(0),\xb_{k}(T)}{\sqrt{\gamma_k}}\right).
\ee
In order to prove that the right hand-side of \eqref{phineg'} is nonpositive, we consider the following first order Taylor expansion of
 $\varphi_i$ around $(\xh(0),\xh(T)):$ 
\benl
\begin{split}
\varphi_i(x_{k}&(0),x_{k}(T))\\
&= 
\varphi_i(\xh(0),\xh(T)) + D \varphi_i(\xh(0),\xh(T)) (\delta x_{k}(0), \delta x_{k}(T)) 
+o(|(\delta x_{k}(0), \delta x_{k}(T))|).
\end{split}
\eenl
Previous equation and Lemmas \ref{lemmaxbar} and \ref{lemmaeta} imply
\benl
\label{taylor_phi}
\varphi_i(x_{k}(0),x_{k}(T))=\varphi_i(\xh(0),\xh(T))+D\varphi_i(\xh(0),\xh(T))(\xb_{k}(0),\xb_{k}(T))+o(\sqrt{\gamma_k}).
\eenl
Thus, the following approximation for the right hand-side of \eqref{phineg'} holds,
\be
\label{difphi}
D\varphi_i(\xh(0),\xh(T)) \left( \frac{\xb_{k}(0),\xb_{k}(T)}{\sqrt{\gamma_k}} \right)
=\frac{\varphi_i(x_{k}(0),x_{k}(T))-\varphi_i(\xh(0),\xh(T))}{\sqrt{\gamma_k}}+o(1).
\ee
Since $w_k\cd$ is a feasible trajectory, it satisfies the final inequality constraint \eqref{finalineq} and, therefore, equations
\eqref{phineg'} and \eqref{difphi} yield, for $1\leq i\leq d_{\varphi},$
$$
D\varphi_i(\xh(0),\xh(T))(\mathring\xi(0),\mathring\xi(T)+B[T]\mathring h)\leq 0.
$$
Now, for $i=0,$ use \eqref{qgrowth} to get the corresponding inequality.
Analogously, one has
\be
\label{eta0}
D\eta_j(\xh(0),\xh(T))(\mathring\xi(0),\mathring\xi(T)+B[T]\mathring h) = 0,\quad \mr{for}\ j=1,\hdots,d_{\eta}.
\ee
Thus $(\mathring\xi\cd,\mathring u\cd,\mathring y\cd,\mathring h)$ satisfies \eqref{tlinearconseq}-\eqref{tlinearconsineq}, and hence it belongs to $\P_2.$


Let us now pass to {Part (B).} 
Notice that from the expansion of $\mathcal{L}$ given in \eqref{taylor0} of Lemma \ref{lemmasc2}, and the inequality \eqref{quadlag} we get
 \be
\Omega_{\P_2} \lam (\mathring\xi_k,\mathring u_k,\mathring y_k,\mathring h_k) \leq o(1),
\ee
and thus
\be 
\label{lims}
\liminf_{k\rightarrow \infty}\, \Omega_{\P_2} \lam (\mathring\xi_k,\mathring u_k,\mathring y_k,\mathring h_k) \leq 0.
\ee
Let us consider the subset of $G({\rm co}\,\Lambda)^{\#}$ defined by
\be
\Lambda^{\#,\rho}:= \{\lambda\in G({\rm co}\,\Lambda)^{\#}: \Omega_{\P_2}\lam - \rho\gamma_\P\ {\rm is}\ {\rm weakly}\,{\rm l.s.c.}\,{\rm on}\ \H_2\times\cR^m \}.
\ee
By applying Lemma \ref{quadform} to the inequality of uniform positivity \eqref{unifpos} one gets
\be
\label{maxLamrho}
\max_{\lambda\in \Lambda^{\#,\rho}} 
 \Omega_{\P_2} \lam (\xib,\ub,\yb,\hb) -
\rho\gamma_\P(\xib(0),\ub,\yb,\hb) \geq 0,\quad \text{on}\
\P_2.
\ee
Let us take the multiplier $\mathring\lambda\in \Lambda^{\#,\rho}$ that attains the maximum in \eqref{maxLamrho} for the direction $(\mathring \xi\cd,\mathring u\cd,\mathring y\cd,\mathring h)$ of $\P_2.$ We get
\be
\begin{split}
0
&\leq 
\Omega_{\P_2}[\mathring\lambda](\mathring \xi,\mathring u,\mathring y,\mathring h) - \rho\gamma_\P(\mathring \xi(0),\mathring u,\mathring y,\mathring h)\\
&\leq
\liminf_{k\rightarrow \infty} \Omega_{\P_2}[\mathring\lambda](\mathring\xi_k,\mathring u_k,\mathring y_k,\mathring h_k) - \rho\gamma_\P(\mathring\xi_{k}(0),\mathring u_k,\mathring y_k,\mathring h_k)
\leq
-\rho,
\end{split}
\ee
since  $\Omega_{\P_2}[\mathring\lambda] - \rho\gamma_\P$ is weakly-l.s.c., $\gamma_\P(\mathring\xi_{k}(0),\mathring u_k,\mathring y_k,\mathring h_k)=1$ for every $k$
and inequality \eqref{lims} holds. 
This leads us to a contradiction since $\rho\gr0.$ Therefore, the desired result follows, this is, the uniform positivity \eqref{unifpos} implies strict weak optimality with $\gamma$-growth.


{\em (ii)} Let us now prove the second statement of the theorem. Assume that $\wh\cd$ is a weak solution satisfying $\gamma$-growth in the weak sense for some constant $\rho'>0,$ and such that $\alpha_0>0$ for every multiplier $\lambda \in G({\rm co}\,\Lambda)^{\#}.$ Let us consider the modified problem
\be\label{tildeP}\tag{$\tilde P$}
\min \{ \varphi_0(x(0),x(T))-\rho' \gamma(x(0)-\xh(0),u\cd-\uh\cd,v\cd-\vh\cd) :  \text{\eqref{stateeq}-\eqref{finalineq}} \},
\ee
and rewrite it in the Mayer form
\be\label{2tildeP}\tag{${\breve P}$}
\begin{split}
\min\,\, &\varphi_0(x(0),x(T))-\rho' \big(|x(0)-\xh(0)|^2 + |y(T)-\yh(T)|^2 +\pi_{1}(T)+\pi_{2}(T)\big),\\
 &\text{\eqref{stateeq}-\eqref{finalineq}},\\
& \dot y=v,\\
& \dot\pi_1 = (u-\uh)^2,\\
& \dot \pi_2 = (y-\yh)^2,\\
& y(0)=0,\ \pi_{1}(0)=0,\ \pi_{2}(0)=0.
\end{split}
\ee
We will next apply the second order necessary condition of Theorem \ref{NCP2} to \eqref{2tildeP} at the point $(w\cd=\wh\cd,y\cd=\yh\cd,\pi_1\cd\equiv0,\pi_2\cd\equiv0).$ Simple computations show that at this solution each critical cone (see \eqref{P}) is the projection of the corresponding critical cone of \eqref{2tildeP}, and that the same holds for the set of multipliers. Furthermore, the second variation of \eqref{2tildeP} evaluated at a multiplier ${\breve \lambda} \in G ({\rm co}\, {\breve \Lambda}^{\#})$ is given by
\be
\Omega_{\P_2}\lam (\xib,\ub,\yb,\yb(T)) - \alpha_0 \rho'\gamma_\P(\xb(0),\ub,\yb,\yb(T)),
\ee
where $\lambda \in G ({\rm co}\, {\Lambda})^{\#}$ is the corresponding multiplier for problem \eqref{P}. 
Hence, the necessary condition in Theorem \ref{NCP2} (see Remark \ref{NCt} below) implies that for every $(\xib\cd,\ub\cd,\vb\cd,\hb)\in \P_2,$ there exists $\lambda \in G ({\rm co}\, {\Lambda})^{\#}$ such that
\benl
\Omega_{\P_2}\lam (\xib,\ub,\yb,\yb(T)) - \alpha_0 \rho'\gamma_\P(\xb(0),\ub,\yb,\yb(T)) \geq 0.
\eenl
Setting $\ds\rho:= \min_{G ({\rm co}\, {\Lambda})^{\#}} \alpha_0 \rho' >0$ the desired result follows.
This completes the proof of the theorem.
\findem

\begin{remark}\label{NCt}
Since the dynamics of \eqref{2tildeP} are not autonomous, what we applied above is an extension of Theorem \ref{NCP2} to time-dependent dynamics. The latter follows easily by adding a state variable $\kappa$ with dynamics $\dot\kappa = 1$ and $\kappa(0)=0.$ 
\end{remark}

\section{Example}\label{SectionExample}

We consider the following example from Dmitruk-Shishov \cite{DmiShi10}:
\be
\label{Pexample}\tag{PE}
\begin{split}
 \min\,\, & -2x_1(T)x_2(T)+x_3(T),\\
 & \dot x_1 = x_2+u,\\
 & \dot x_2 = v,\\
 & \dot x_3 = x_1^2 + x_2^2 + x_2 v+u^2\\
 & x_1(0) = x_2(0) = x_3(0) =0.
 \end{split}
 \ee
 Let us use $p_1,p_2,p_3$ to denote the costate variables associated to \eqref{Pexample}. Observe that $\dot p_3\cd\equiv 0$ and $p_3(T)=1,$ thus  $p_3\cd\equiv 1.$ Note as well that the linearized state equation implies $\dot\xb_2=\vb,\, \xb_1(0)=\xb_2(0)=\xb_3(0)=0.$ Consequently, $\yb\cd=\xb_2\cd,$  $\xib_1(0)=\xib_2(0)=\xib_3(0)=0,$ and
 $$
 \xib_2\cd=\xb_2\cd-\yb\cd \equiv 0,
 $$ 
where the first equality follows from Goh's transformation \eqref{Goht}.

Recalling the definitions given in \eqref{M}-\eqref{g}, the second variation $\Omega_{\P_2}$ (defined in \eqref{OmegaP2}) on the critical cone $\P_2$ of \eqref{Pexample} gives:
\be
\Omega_{\P_2}(\xb,\ub,\yb,\hb) =  \half \hb^2+\int_0^T (\xb_1^2+\ub^2+\yb^2)\dtt.
\ee
We see that $\Omega_{\P_2}$ verifies the sufficient condition \eqref{unifpos}. We should now look for a feasible solution that verifies the first order optimality conditions.

In Aronna \cite{Aro13} we used the {\em shooting algorithm}  to solve problem \eqref{Pexample} numerically. The numerical tests converged to the optimal solution $(\uh,\vh)\cd\equiv 0$ for arbitrary guesses of the initial values of the costate variables. It is inmediate to check that 
$\wh\cd\equiv 0$ is a feasible trajectory that verifies the first order optimality conditions. Since the second variation at this $\wh$ verifies the sufficient condition of Theorem \ref{SC}, we conclude that $\wh\cd$ is a strict weak optimal trajectory that satisfies $\gamma$-growth.

\section{Conclusion and possible extensions}\label{SectionConclusion}

We studied optimal control problems in the Mayer form governed by systems that are affine in some components of the control variable. 
A set of `no gap' necessary and sufficient second order  optimality conditions was provided. These conditions apply to a weak minimum, consider fairly general endpoint constraints and do not assume  uniqueness of multiplier.
We further derived the Goh conditions when we assume uniqueness of multiplier.

The main result of the article is Theorem \ref{SC}. The interest of this result is that it can be applied either to prove optimality of some candidate solution of a given problem, or to show convergence of an associated shooting algorithm as stated in Aronna \cite{Aro13} and proved in the detail in the technical report Aronna \cite{Aro11}. This algorithm and its proof of convergence apply also to partially-affine problems with bounds on the control and bang-singular solutions, and hence its convergence has strong  practical interest.

The results here presented can be pursued by many interesting extensions. One of the most important extensions are the optimality conditions for bang-singular solutions for problems containing closed control constraints.

\section*{Acknowledgments}
Part of this work was done during my Ph.D. 
under the supervision of
Fr\'ed\'eric Bonnans, who I thank for the great guidance.

I also acknowledge the anonymous referee for his careful reading and useful remarks.

\appendix

\section{Proofs of technical results}
We include in this part the proofs that were omitted throughout the article.

\subsection{}\label{proofexpansionlagrangian}

\noindent{\em Proof of Lemma \ref{expansionlagrangian}.}
We shall omit the dependence on $\lambda$ for the sake of
simplicity of notation. 
Let us consider the following second
order Taylor expansions, written in a compact form, 
\begin{gather} 
\label{expell} 
 \ell (x(0),x(T))
= \ell + {D\ell} (\delta x(0),\delta x(T))
+ \half {D^2\ell} (\delta x(0),\delta x(T))^2
+ L_\ell|(\delta x(0),\delta x(T))|^3, 
\\
\label{expfi}
 f_i(x,u) = f_{i} + {D f_{i}} (\delta x,\delta u)
+ \half {D^2 f_{i}} (\delta x,\delta u)^2
+L |(\delta x,\delta u) |^3.
\end{gather}
Observe that, in view of the transversality conditions \eqref{transvcond} and the costate equation \eqref{costateeq}, one has
\be
\label{Dell}
\begin{split}
{D\ell}\,&(\delta x(0),\delta x(T)) 
= -p(0)\, \delta x(0) + p(T) \, \delta x(T) \\
&= \intT \left[ \dot{p}\, \delta x + p \dot{\delta x} \right] \dtt =
\intT p\,\Big[-\big({D_xf_i}+\sum_{i=1}^m\vh_i {D_xf_i}\big)  \delta x +\dot{\delta x}\Big] \dtt.
\end{split}
\ee
In the definition of $\mathcal{L}$ given in \eqref{lagrangian}, replace $\ell(x(0),x(T))$ and $f_i(x,u)$  by their Taylor
expansions \eqref{expell}-\eqref{expfi} and use the identity \eqref{Dell}. This  yields
\begin{align*}
\mathcal{L}(w) 
=&\,\,\mathcal{L}(\wh) + \ds \intT \big[ H_u\delta u + H_v\delta v \big]\dtt + \Omega(\delta x,\delta
u,\delta v) \\
&+ \ds \intT \big[ H_{vxx}(\delta x, \delta x,\delta v) 
+ 2H_{vux}(\delta x,\delta u,\delta v) 
+ H_{vuu}(\delta u,\delta u,\delta v) \big] \dtt \\
&+ \, L_\ell |(\delta x(0),\delta x(T))|^3
+ L(1+\|v\|_\infty)\,\|(\delta x,\delta u) \|_{\infty} \ds\intT p\,|(\delta x,\delta u) |^2 \dtt.
\end{align*}
Finally, to obtain \eqref{expansionLag}, remove the first order terms by the stationarity conditions \eqref{stationarity}, and use the Cauchy-Schwarz inequality in the last integral.
This completes the proof.
\findem

\subsection{Proof of Theorem \ref{classicalNC}}\label{AppendixclassicalNC}

Let us write problem (P) in an {\em abstract} form defining, for $j=1,\dots,d_\eta$ and $i=0,\dots,d_\varphi,$
\begin{gather*}
\bar{\eta}_j:\cR^n\times \U \times \V \rightarrow \cR,\quad (x(0),u\cd,v\cd)\mapsto \bar{\eta}_j(x(0),u\cd,v\cd):=\eta_j(x(0),x(T)),\\
\bar{\varphi}_i:\cR^n\times \U \times \V \rightarrow \cR,\quad (x(0),u\cd,v\cd)\mapsto \bar{\varphi}_i(x(0),u\cd,v\cd):=\varphi_i(x(0),x(T)),
\end{gather*}
where $x\cd\in\X$ is the solution of \eqref{stateeq} associated with $(x(0),u\cd,v\cd).$ Hence, (P) can be written as the following problem in the space $\cR^n\times \U \times \V,$
\be
\label{AP}\tag{AP}
\begin{split}
\min\,\,&\bar\varphi_0(x(0),u,v);\\
\,\,\text{s.t.}\,\,
&\bar\eta_j(x(0),u,v)=0,\ \text{for}\ j=1,\dots,d_{\eta},\\
&\bar\varphi_i(x(0),u,v)\leq 0,\ \text{for}\ j=1,\dots,d_{\varphi}.
\end{split}
\ee
Notice that if $\wh\cd$ is a weak solution of (P) then $(\xh(0),\uh\cd,\vh\cd)$ is a local solution of (AP). 

\begin{definition} 
\label{DefCQ}
We say that the endpoint equality constraints are {\it qualified} if
\be\label{QC}
D\bar\eta(\xh(0),\uh,\vh)\ \text{is onto from}\ \cR^n\times \U\times\V\ \text{to}\ \cR^{d_{\eta}}.
\ee
When \eqref{QC} does not hold, the constraints are {\em not qualified} or {\em unqualified}.
\end{definition}

The proof of Theorem \ref{classicalNC} is divided in two cases: qualified and not qualified endpoint equality constraints. In the latter case the condition \eqref{classicalNCeq} follows easily and it is shown in Lemma \ref{degNC} below. The proof for the qualified case is done by means of an auxiliary  linear problem and duality arguments.

\begin{lemma}\label{degNC}
If the equality constraints are not qualified then \eqref{classicalNCeq} holds.
\end{lemma}

\begin{proof}
Observe that since $D\bar\eta(\xh(0),\uh,\vh)$ is not onto there exists $\beta\in\cR^{d_{\eta},*}$ with $|\beta|=1$ such that $\sum_{j=1}^{d_{\eta}} \beta_j D\bar\eta_j(\xh(0),\uh,\vh)=0$ and consequently, 
\benl
\sum_{j=1}^{d_{\eta}} \beta_j D\eta_j(\xh(0),\xh(T))=0.
\eenl
 Set $\lambda:=(p\cd,\alpha,\beta)$ with $p\cd\equiv 0$ and $\alpha=0.$ Then both $\lambda$ and $-\lambda$ are in $\Lambda.$ Observe that \benl
\Omega\lam(\xb,\ub,\vb)=
\half \sum_{j=1}^{d_{\eta}} \beta_j D^2\eta_j(\xh(0),\xh(T))(\xb(0),\xb(T))^2.
\eenl
Thus, either $\Omega\lam(\xb,\ub,\vb)$ or $\Omega[-\lambda](\xb,\ub,\vb)$  is necessarily nonnegative. The desired result follows.
\end{proof}

Let us now deal with the qualified case. Take a critical direction $\wb\cd=(\xb,\ub,\vb)\cd\in \C$ and consider the problem in the variables $\tau\in\cR$ and $r=(r_{x_0},r_u,r_v)\in\cR^n\times \U\times \V$ given by
\be\label{QPw}\tag{QP$_{\wb}$}
\begin{split}
\min\,\, &\tau \\
\text{s.t.}\,\,\, & D\bar\eta(\xh(0),\uh,\vh) r +D^2\bar\eta(\xh(0),\uh,\vh)(\xb(0),\ub,\vb)^2 =0,\\
& D\bar\varphi_i(\xh(0),\uh,\vh) r +D^2\bar\varphi_i(\xh(0),\uh,\vh)(\xb(0),\ub,\vb)^2 \leq \tau,\ \ i=0,\dots,d_{\varphi}.
\end{split}
\ee

\begin{proposition}\label{dualpb}
Assume that $\wh\cd$ is a weak solution of \eqref{AP} for which the endpoint equality constraints are qualified. Let $\wb\cd\in \C$ be a critical direction. Then the problem  \eqref{QPw} is feasible and has nonnegative value.
\end{proposition}

\noindent{\em Proof of Proposition \ref{dualpb}.}
Step I. Let us first show feasibility. Since $D\bar\eta(\xh(0),\uh,\vh)$ is onto, there exists $r\in \cR^n\times \U\times \V$ for which the equality constraint in \eqref{QPw} is satisfied. Set
\be
\label{zeta1}
\tau:=\max_{0\leq i \leq d_{\varphi}}\{D\bar\varphi_i(\xh(0),\uh,\vh) r +D^2\bar\varphi_i(\xh(0),\uh,\vh)(\xb(0),\ub,\vb)^2 \}.
\ee
Then $(\tau,r)$ is feasible for \eqref{QPw}.

Step II. Let us now prove that \eqref{QPw} has nonnegative value. Suppose on the contrary that there is $(\tau,r)\in \cR\times\cR^n\times \U\times \V$ feasible for \eqref{QPw} with $\tau<0.$ 
We shall look for a family of feasible solutions of \eqref{AP} referred as $\{r(\sigma)\}_{\sigma}$ with the following properties:  it is defined for small positive values of $\sigma$ and  it satisfies 
\be\label{estrsigma}
r(\sigma)\underset{\sigma\to0}{\longrightarrow} (\xh(0),\uh,\vh)\ \text{in} \  \cR^n\times \U\times \V,\ \text{and}\ 
\bar\varphi_0(r(\sigma) )< \bar\varphi_0(\xh(0),\uh,\vh).
\ee
The existence of such family $\{r(\sigma)\}_{\sigma}$ will contradict the local optimality of $(\xh(0),\uh,\vh).$
Consider hence
\benl
\tilde{r}(\sigma):= (\xh(0),\uh,\vh)+\sigma(\xb(0),\ub,\vb)+\half\sigma^2 r.
\eenl
Let $0\leq i\leq d_{\varphi}$ and observe that 
\be\label{estvarphi}
\begin{split}
\bar\varphi_i(\tilde{r}(\sigma))
=&
\,\bar\varphi_i(\xh(0),\uh,\vh)+\sigma D\bar\varphi_i(\xh(0),\uh,\vh)(\xb(0),\ub,\vb)\\
&+\half\sigma^2\left[ D\bar\varphi_i(\xh(0),\uh,\vh)r+D^2 \bar\varphi_i(\xh(0),\uh,\vh)(\xb(0),\ub,\vb)^2 \right] + o(\sigma^2)\\
\leq&\,\bar\varphi_i(\xh(0),\uh,\vh)+\half\sigma^2\tau+o(\sigma^2),
\end{split}
\ee
where last inequality holds since $(\xb,\ub,\vb)\cd$ is a critical direction and in view of the definition of $\tau$  in \eqref{zeta1}.
Analogously, one has
\benl
\bar\eta(\tilde{r}(\sigma))=o(\sigma^2).
\eenl
Since $D\bar\eta(\xh(0),\uh,\vh)$ is onto, there exists $r(\sigma) \in \cR\times \U\times \V$ such that $\| r(\sigma)-\tilde{r}(\sigma)\|_{\infty}=o(\sigma^2)$ and $\bar\eta(r(\sigma))=0.$
This follows by applying the Implicit Function Theorem to the mapping
\benl
(r,\sigma)\mapsto \bar\eta\left((\xh(0),\uh,\vh)+\sigma(\xb(0),\ub,\vb)+\half\sigma^2 r\right)=\bar\eta(\tilde{r}(\sigma)).
\eenl
On the other hand, by taking $\sigma$ sufficiently small 
in estimate \eqref{estvarphi}, we obtain 
\benl
\bar\varphi_i({r}(\sigma))<\bar\varphi_i(\xh(0),\uh,\vh),
\eenl
since $\tau<0.$ Hence $r(\sigma)$ is feasible for \eqref{AP} and  verifies \eqref{estrsigma}. This contradicts the optimality of $(\xh(0),\uh,\vh).$ We conclude then that all the feasible solutions of \eqref{QPw} have $\tau\geq 0$ and, therefore, its value is nonnegative.

\findem

We shall now proceed to prove Theorem \ref{classicalNC}.

\noindent{\em Proof of Theorem \ref{classicalNC}.}
The unqualified case is covered by Lemma \ref{degNC} above. Hence, for this proof, assume that \eqref{QC} holds.

Given $\wb\cd\in \C, $ note that \eqref{QPw} can be regarded as a linear problem in the variables $(\zeta, r),$ whose associated dual  is given by 
\begin{align}
\label{dualQP1}
\max_{(\alpha,\beta)}\,\,\, &\sum_{i=0}^{d_{\varphi}} \alpha_i D^2\bar\varphi_i(\xh(0),\uh,\vh)(\xb(0),\ub,\vb)^2 + \sum_{j=1}^{d_{\eta}} \beta_j D^2\bar\eta_j(\xh(0),\uh,\vh)(\xb(0),\ub,\vb)^2 \\
\label{dualQP2}
{\rm s.t.}\,\,\, & \sum_{i=0}^{d_{\varphi}} \alpha_i D\bar\varphi_i(\xh(0),\uh,\vh)+\sum_{j=1}^{d_{\eta}} \beta_j D\bar\eta_j(\xh(0),\uh,\vh)=0,\\
\label{dualQP3}
&\sum_{i=0}^{d_{\varphi}} \alpha_i=1,\quad \alpha \geq 0.
\end{align}
The Proposition \ref{dualpb} above and the linear duality result  Bonnans \cite[Theorem 3.43]{BonOC} imply that \eqref{dualQP1}-\eqref{dualQP3} has finite nonnegative value  (the reader is referred to Shapiro \cite{Sha01} and references therein for a general theory on linear duality). Consequently, there exists a feasible solution $(\bar\alpha,\bar\beta)\in \cR^{d_\varphi+d_\eta+1}$ to \eqref{dualQP1}-\eqref{dualQP3}, with associated nonnegative and finite value. Set $(\alpha,\beta) := (\bar\alpha,\bar\beta)/(\sum_{i=0}^{d_{\varphi}} |\bar\alpha_i|+\sum_{j=1}^{d_{\eta}} |\bar\beta_j|),$ where the denominator is not zero in view of \eqref{dualQP3}. We get that $(\alpha,\beta)\in \cR^{d_\varphi+d_\eta+1}$ verifies \eqref{nontriv}-\eqref{alphapos}, \eqref{dualQP2} and 
\be
\label{DM2}
\sum_{j=1}^{d_{\eta}} \beta_j D^2\bar\eta_j(\xh(0),\uh,\vh)(\xb(0),\ub,\vb)^2+ \sum_{i=0}^{d_{\varphi}} \alpha_i D^2\bar\varphi_i(\xh(0),\uh,\vh)(\xb(0),\ub,\vb)^2 \geq0.
\ee
\if{
implies that there cannot exist $(\tau,r)\in \cR\times \cR^n\times \U\times \V$ such that
\benl
\left\{
\begin{split}
&D\bar\eta(\xh(0),\uh,\vh) r +D^2\bar\eta(\xh(0),\uh,\vh)(\xb(0),\ub,\vb)^2 =0,\\
&D\bar\varphi_i(\xh(0),\uh,\vh) r +D^2\bar\varphi_i(\xh(0),\uh,\vh)(\xb(0),\ub,\vb)^2 \leq \tau,\ \text{for}\ i=0,\dots,d_{\varphi},\\
&\tau <0.
\end{split}
\right.
\eenl
Therefore, the Dubovitskii-Milyutin Theorem \cite{DubMil} guarantees the existence of $(\alpha,\beta)\in\cR^{1+d_{\varphi}+d_{\eta}}$ such that $(\alpha,\beta)\neq0,$ $\alpha\geq 0$ and 
\begin{gather}
\label{DM1}
\ds\sum_{i=0}^{d_{\varphi}} \alpha_i D\bar\varphi_i(\xh(0),\uh,\vh) + \sum_{j=1}^{d_{\eta}} \beta_j D\bar\eta_j(\xh(0),\uh,\vh) =0,\\
\label{DM2}
\ds\sum_{i=0}^{d_{\varphi}}\alpha_i D^2\bar\varphi_i(\xh(0),\uh,\vh)(\xb(0),\ub,\vb)^2 + \sum_{j=1}^{d_{\eta}} \beta_j D^2\bar\eta_j(\xh(0),\uh,\vh)(\xb(0),\ub,\vb)^2 \geq 0
\end{gather}
}\fi
For this $(\alpha,\beta),$ let $p\cd$ be the solution of \eqref{costateeq}
with final condition
\be
\label{DM5}
p(T)=\sum_{i=0}^{d_{\varphi}} \alpha_i D_{x_T}\varphi_i(\xh(0),\xh(T)) + \sum_{j=1}^{d_{\eta}} \beta_j D_{x_T}\eta_j(\xh(0),\xh(T)).\ee
We shall prove that $\lambda := (\alpha,\beta,p\cd)$ is in $\Lambda,$ i.e. that also the first line in \eqref{transvcond} and the stationarity conditions \eqref{stationarity} hold. Let $(\tilde{x},\tilde{u},\tilde{v})\cd\in\W$ be the solution of the linearized state equation \eqref{lineareq}. In view of \eqref{dualQP2},
\be
\label{DM4}
\sum_{i=0}^{d_{\varphi}} \alpha_i D\bar\varphi_i(\xh(0),\uh,\vh)(\tilde{x}(0),\tilde{u},\tilde{v}) + \sum_{j=1}^{d_{\eta}} \beta_j D\bar\eta_j(\xh(0),\uh,\vh)(\tilde{x}(0),\tilde{u},\tilde{v}) =0,
\ee
Hence, rewriting in terms of the endpoint Lagrangian $\ell$ and using \eqref{DM5}-\eqref{DM4}, one has
\benl
\label{DM3}
0 = D\ell \lam (\xh(0),\xh(T))(\tilde{x}(0),\tilde{x}(T)) = D_{x_0}\ell \lam (\xh(0),\xh(T))\tilde{x}(0) + p(T) \tilde{x}(T) \pm p(0) \tilde{x}(0).
\eenl
By regrouping terms in the previous equation, we get
\be
\label{DM6}
\begin{split}
0&= \Big( D_{x_0}\ell\lam(\xh(0),\xh(T))+p(0)\Big) \tilde{x}(0) + \intT (\dot{p}\tilde{x} +p \dot{\tilde{x}} ) \dtt \\
&= \Big( D_{x_0}\ell\lam(\xh(0),\xh(T))+p(0)\Big) \tilde{x}(0) + \intT \big( H_u\lam \tilde{u} + H_v\lam \tilde{v} \big) \dtt,
\end{split}
\ee
where we used  \eqref{costateeq} and \eqref{lineareq} in the last equality. Since \eqref{DM6} holds for all $(\tilde{x}(0),\tilde{u}\cd,\tilde{v}\cd)$ in $\cR^n \times \U\times \V,$ the first line in \eqref{transvcond} and the stationarity conditions in \eqref{stationarity} are necessarily verified.
Thus, $\lambda$ is an element of $\Lambda.$ On the other hand, simple computations yield that  \eqref{DM2} is equivalent to
\benl
\Omega\lam (\xb,\ub,\vb)\geq0,
\eenl
and, therefore, the result follows.

\findem

\subsection{} {\em Proof of Lemma \ref{Omegat}.}\label{AppendixOmegat}
First recall that the term $\vb\tras H_{vu}\lam \ub$ in $\Omega\lam$ vanishes since we are taking $\lambda\in \Lambda^{\#}$ and, in view of Lemma \ref{Lambdawlsc}, $H_{vu}\lam \equiv 0.$
In the remainder of the proof we omit the dependence on $\lambda$ for the sake of
simplicity.
Replacing $\xb$ in the definition of $\Omega$ in equation  \eqref{Omega} by its expression in
\eqref{Goht} yields
\be
\label{J2}
\begin{split}
\Omega & (\xb,\ub,\vb) =\\
&\, \half\ell''(\xh(0),\xh(T))\big(\xib(0),\xib(T)+F_{v}[T]\,\yb(T)\big)^2
+ \ds \int_0^T \left[ \half(\xib+F_v\,\yb)\tras H_{xx}(\xib+F_v\,\yb) 
\right. \\
& \left. +\,\ub\tras H_{ux} (\xib+F_v\,\yb)  + \vb\tras H_{vx} (\xib+F_v\,\yb) + \half\ub\tras H_{uu}\,\ub \right] \dtt.
\end{split}
\ee
In view of \eqref{xieq} one gets
\be
\label{J2.2.1}
\int_0^T \vb\tras H_{vx}\, \xib \dtt
=[\yb\tras H_{vx}\, \xib]_0^T -\int_0^T \yb\tras
\{\dot{H}_{vx}\,\xib+H_{vx}(F_x\,\xib+F_u\,\ub+B\,\yb)\} \dtt.
\ee
The decomposition of $H_{vx}\,F_v$ introduced  in \eqref{SV} followed by an integration by parts leads to  
\be\label{J2.2.2}
\begin{split}
\int_0^T \vb\tras H_{vx}\,F_v \yb \dtt
&=\int_0^T \vb\tras (S+G) \yb \dtt\\
&=\half[\yb\tras S\yb]_0^T+\int_0^T(-\half \yb\tras
\dot S\yb + \vb\tras G\yb )\dtt.
\end{split}
\ee
The result follows by replacing using \eqref{J2.2.1} and \eqref{J2.2.2} in \eqref{J2}.

\findem

\section{Technical lemmas used in the proof of the main Theorem \ref{SC}}\label{AppendixSC}

Recall first the following classical result for ordinary differential equations.
\begin{lemma}[Gronwall's Lemma]
\label{GronLem}
Let $a\cd\in W^{1,1}(0,T;\cR^n),$ $b\cd\in L^1(0,T)$ and $c\cd\in L^1(0,T)$ be such that $|\dot{a}(t)|\leq b(t) + c(t) |a(t)|$ for a.a. $t\in (0,T).$ Then
\benl
\|a\cd\|_\infty \leq e^{\|c\cd\|_1}\big(|a(0)|+\|b\cd\|_1 \big).
\eenl
\end{lemma}

For the lemma below recall the definition of the space $\H_2$ given in \eqref{H2}.

\begin{lemma}
\label{lemmaxbar}
There exists $\rho\gr 0$ such that 
\be
\label{xbargamma}
|\xb(0)|^2+\|\xb\cd\|_2^2+|\xb(T)|^2\leq \rho
\gamma(\xb(0),\ub\cd,\vb\cd),
\ee
for every linearized trajectory $(\xb,\ub,\vb)\cd\in\H_2.$ The constant $\rho$ depends on $\|A\|_{\infty},$
$\|F_v\|_{\infty},$ $\|E\|_{\infty}$ and
$\|B\|_{\infty}.$
\end{lemma}

\begin{proof}
Throughout this proof, whenever we put $\rho_i$ we refer to a positive constant depending on $\|A\|_{\infty},$
$\|F_v\|_{\infty},$ $\|E\|_{\infty},$ and/or
$\|B\|_{\infty}.$
Let $(\xb,\ub,\vb)\cd\in\H_2$ and 
$(\xib,\yb)\cd$ be defined by Goh's Transformation \eqref{Goht}.
 Thus $(\xib,\ub,\yb)\cd$ is solution of \eqref{xieq}. Gronwall's Lemma \ref{GronLem} and
Cauchy-Schwarz inequality yield
\be\label{lemmazxit}
\|\xib\|_{\infty}\leq \rho_1
(|\xib(0)|^2+\|\ub\|_2^2+\|\yb\|_2^2)^{1/2}\leq
\rho_1 \gamma_\P(\xb(0),\ub,\yb,\yb(T))^{1/2},
\ee
with
$\rho_1=\rho_1(\|A\|_1,\|E\|_{\infty},\|B\|_{\infty}).$
This last inequality together with the relation between $\xib\cd$ and $\xb\cd$ provided by \eqref{Goht}
imply
\be\label{lemmazz}
\|\xb\|_2\leq \|\xib\|_2+\|F_v\|_{\infty}\|\yb\|_2\leq\rho_2
\gamma_\P(\xb(0),\ub,\yb,\yb(T))^{1/2},
\ee
for $\rho_2=\rho_2(\rho_1,\|F_v\|_{\infty}).$
On the other hand,  \eqref{Goht} and estimate
\eqref{lemmazxit} lead to
\benl
|\xb(T)|\leq |\xib(T)|+\|F_v\|_{\infty}|\yb(T)|\leq
\rho_1 \gamma_\P(\xb(0),\ub,\yb,\yb(T))^{1/2}
+\|F_v\|_{\infty}|\yb(T)|.
\eenl
Then, in view of Young's inequality `$2{ab}\leq {a^2+b^2}
$' for real numbers $a,b,$ one gets 
\be\label{lemmazzT}
|\xb(T)|^2\leq
\rho_3 \gamma_\P(\xb(0),\ub,\yb,\yb(T)),
\ee
for some
$\rho_3=\rho_3(\rho_1,\|F_v\|_{\infty}).$
The desired estimate follows from \eqref{lemmazz}
and \eqref{lemmazzT}.
\end{proof}
Notice that Lemma \ref{lemmaxbar} above gives an estimate of the linearized state in the order $\gamma.$ The following result shows that the analogous property holds for the variation of the state variable as well and it is a natural extension of a similar result given in Dmitruk \cite{Dmi87} for control-affine systems. 

\begin{lemma} \label{lemmadeltax}
Given $C\gr 0,$ there exists $\rho\gr 0$ such that 
\be
|\delta x(0)|^2+\|\delta x\cd\|^2_2+|\delta x(T)|^2\leq \rho
\gamma(\delta x(0),\delta u\cd,\delta v\cd),
\ee
for every $(x,u,v)\cd$ solution of the state equation \eqref{stateeq} having $\|v\cd\|_2\leq C,$ and where  $\delta w\cd:=w\cd-\wh\cd.$ The constant $\rho$ depends on $C,$ $\|B\|_{\infty},$ $\|\dot B\|_{\infty}$ and the Lipschitz constants of $f_i.$
\end{lemma}

\begin{proof}
In order to simplify the notation we omit the dependence on $t.$
Consider $(x,u,v)\cd$ solution of \eqref{stateeq} with
$\|v\cd\|_2\leq C.$ Let $\delta w\cd:=w\cd-\wh\cd,$ $\delta y(t):=\int_0^t \delta v(s) {\rm d}s,$ and
$\xi\cd:=\delta x\cd-B[\cdot]\delta y\cd,$ with $y(t):=\int_0^t v(s)\mr{d}s.$ Note
that
\be \label{xidot}
\begin{split}
\dot\xi=
& f_0(x,u)-f_0(\xh,\uh)+\sum_{i=1}^m \left[v_if_i(x,u)-\vh_if_i(\xh,\uh)\right]-\dot B\delta y- \sum_{i=1}^m \delta
v_i\, f_i(\xh,\uh)\\
=&f_0(x,u)-f_0(\xh,\uh)+\sum_{i=1}^m v_i[f_i(x,u)-f_i(\xh,\uh)]-\dot{B}\delta y.
\end{split}
\ee
In view of the Lipschitz-continuity of $f_i,$
\be
|f_i(x,u)-f_i(\xh,\uh)|\leq L (|\delta x|+|\delta u|) \leq
L(|\xi|+\|B\|_{\infty}|\delta y|+|\delta u|),
\ee
for some $L\gr 0.$
Thus, from \eqref{xidot} it follows
\benl
|\dot\xi|
\leq L(|\xi|+\|B\|_{\infty}|\delta
y|+|\delta u|)(1+|v|)+\|\dot B\|_{\infty}|\delta y|.
\eenl
Applying Gronwall's Lemma \ref{GronLem} one gets
\benl
\| \xi\|_\infty \leq e^{L\|1+|v|\,\|_1}\Big[ |\xi(0)| + \left\|L(1+|v|)(\|F_v\|_\infty |\delta y| + |\delta u|)+\|\dot{F}_v\|_\infty |\delta y|\,\right\|_1  \Big],
\eenl
and Cauchy-Schwarz inequality applied to previous estimate yields
\be
\|\xi\|_{\infty} 
\leq \rho_1\big(|\xi(0)|+\|\delta y\|_1+\|\delta u\|_1
+ \|\delta y\|_2\|v\|_2 +  \|\delta
u\|_2\|v\|_2\big),
\ee
for $\rho_1=\rho_1(L,C,\|F_v\|_{\infty},\|\dot{F}_v\|_{\infty}).$
Since $\|\delta x\|_2\leq
\|\xi\|_2+\|F_v\|_{\infty}\|\delta y\|_2,$ by  previous estimate and Cauchy-Schwarz inequality,  the result follows.
\end{proof}

Finally, the following lemma gives an estimate for the difference between the variation of the state variable and the linearized state.

\begin{lemma}
\label{lemmaeta}
Consider $C\gr 0$ and $w\cd=(x,u,v)\cd\in\W$ a trajectory with $\|w\cd-\wh\cd \|_{\infty}\leq C.$ 
Set $(\delta x,\delta u,\delta v)\cd:= w\cd-\wh\cd$ and let $\xb\cd$ be the linearization of $\xh\cd$ associated with $(\delta x,\delta u,\delta v)\cd.$
Define
\be
\vartheta\cd:=\delta x\cd-\xb\cd.
\ee
Then, $\vartheta\cd$ is solution of the differential equation
\be
\label{doteta}
\begin{split}
\dot\vartheta
&= 
D_xf_{0} (\xh,\uh)\vartheta +\sum_{i=1}^m \vh_i D_xf_{i} (\xh,\uh)\vartheta 
+ 
\sum_{i=1}^m \delta v_i Df_{i} (\xh,\uh)(\delta x,\ub) + \zeta,\\
\vartheta(0) &= 0,
\end{split}
\ee
where the remainder $\zeta\cd$ is given by
\be 
\label{zeta}
\zeta:= \half D^2f_0(\xh,\uh)(\delta x,\ub)^2+\sum_{i=1}^m \half v_i  D^2f_i(\xh,\uh)(\delta x,\ub)^2+L \left(1+\sum_{i=1}^m v_i\right)|(\delta x,\ub)|^3,
\ee
and $L$ is a Lipschitz constant for $D^2f_i,$ uniformly in $i=0,\dots,m.$ Furthermore, $\zeta\cd$ satisfies the estimates
\be
\label{estzeta}
\|\zeta\cd\|_{\infty} \mi \rho_1C,\quad \|\zeta\cd\|_2 \mi \rho_1C\sqrt\gamma,
\ee
where $\rho_1= \rho_1(C,\|D^2 f\|_{\infty},L,\|v\|_\infty+1).$

If in addition, $C\rightarrow 0,$
the following estimates for $\vartheta\cd$ hold
\be
\label{esteta}
\|\vartheta\cd\|_{\infty} = o(\sqrt\gamma),
\quad 
\|\dot\vartheta\cd\|_2 = o(\sqrt\gamma).
\ee
\end{lemma}

\begin{proof}
We shall note first that
\be
\label{dotdeltax1}
\dot{\delta x} =
f_0(x,u)-f_0(\xh,\uh)+\sum_{i=1}^m v_i \big[f_i(x,u)-f_i(\xh,\uh) \big] + \sum_{i=1}^m \delta v_{i}\,f_i(\xh,\uh).
\ee
Consider the following second order Taylor expansions for $f_i,$
\be
\label{Taylorfi}
f_i(x,u)= f_i(\xh,\uh) + D f_i(\xh,\uh)(\delta x,\delta u) + \half D^2f_i(\xh,\uh)(\delta x,\delta u)^2 + {L}|(\delta x,\delta u)|^3.
\ee
Combining \eqref{dotdeltax1} and \eqref{Taylorfi} yields
\be
\label{dotdeltax}
\dot{\delta x} =
D f_0(\xh,\uh)(\delta x,\delta u)+\sum_{i=1}^m v_i D f_i(\xh,\uh)(\delta x,\delta u)+\sum_{i=1}^m \delta v_{i}f_i(\xh,\uh)+\zeta,
\ee
with the remainder being given by \eqref{zeta}.
The linearized equation \eqref{lineareq} together with \eqref{dotdeltax} lead to \eqref{doteta}.
In view of \eqref{zeta} and Lemma \ref{lemmadeltax}, it can be seen that the estimates in \eqref{estzeta} hold.

On the other hand, by applying Gronwall's Lemma \ref{GronLem} to \eqref{doteta}, and using Cauchy-Schwarz inequality afterwards lead to
\benl
\|\vartheta\|_{\infty} 
\leq 
\rho_3\left\|\sum_{i=1}^m \delta v_i D f_{i} (\xh,\uh) (\delta x,\delta u) + \zeta\right\|_1 
\leq 
\rho_4 \Big[ \|\delta v\|_2(\|\delta x\|_2 + \|\delta u\|_2) + \|\zeta\|_2 \Big],
\eenl
for some positive $\rho_3,\rho_4$ depending on $\|\vh\|_{\infty}$ and $\|Df\|_{\infty}.$ 
Finally, using the estimate in Lemma \ref{lemmadeltax} and \eqref{estzeta} just obtained, the inequalities in \eqref{esteta} follow.
\end{proof}

In view of Lemmas \ref{expansionlagrangian}, \ref{lemmaxbar}, \ref{lemmadeltax} and \ref{lemmaeta} we can justify 
the following technical result that is an essential point in the proof of the sufficient condition of Theorem \ref{SC}.

\begin{lemma}
\label{lemmasc2}
Let $w\cd\in\W$ be a trajectory.
Set $(\delta x,\delta u,\delta v)\cd:= w\cd-\wh\cd,$ and $\xb\cd$ its corresponding linearized state, i.e. the solution of \eqref{lineareq}-\eqref{lineareq0} associated with $(\delta x(0),\delta u\cd,\delta v\cd).$ Assume that    $\|w\cd-\wh\cd\|_{\infty} \rightarrow 0.$
Then  
\be
\label{taylor0}
\mathcal{L}[\lambda](w) = \mathcal{L}[\lambda](\wh) 
+ \Omega[\lambda](\xb,\delta u,\delta v)+o(\gamma),
\ee
for every $\lambda \in {\rm co}\, \Lambda.$
\end{lemma}

\begin{proof}
For the sake of simplicity of notation, we shall omit the dependence on $\lambda.$ 

Let us recall the expansion of the Lagrangian function given in Lemma \ref{expansionlagrangian}, and observe that it also holds for any $\lambda$ in ${\rm co}\, \Lambda.$ Next,  
notice that, by Lemma \ref{lemmadeltax}, 
$
\mathcal{L}(w)
=
\mathcal{L}(\wh)+ \Omega(\delta x,\delta u,\delta v)+o(\gamma).
$
Hence, 
\be
\label{taylorlemma}
\mathcal{L}(w)
=
\mathcal{L}(\wh)+ \Omega(\xb,\delta u,\delta v)+\Delta\Omega + o(\gamma),
\ee
with 
$
\Delta\Omega:= \Omega(\delta x,\delta u,\delta v)-\Omega(\xb,\delta u,\delta v).
$
The next step is using Lemmas \ref{lemmaxbar}, \ref{lemmadeltax} and \ref{lemmaeta} to  prove that
\begin{equation}
\label{rest}
\Delta \Omega=o(\gamma).
\end{equation}
Note  that $\Q(a,a)-\Q(b,b)=\Q(a+b,a-b),$ for any bilinear mapping $\Q,$ and any pair $a,b$ of elements in its domain. Set $\vartheta\cd:=\delta x\cd-  \xb\cd$ as it is done in Lemma \ref{lemmaeta}. 
Hence, 
\benl
\begin{split}
\Delta \Omega
= &\,
\half\ell'' \Big( (\delta x(0)+\xb(0),\delta x(T)+\xb(T)) , (0,\vartheta(T)) \Big) \\
& +\ds\int_0^T [\half(\delta x+\xb)\tras Q\vartheta + \delta u\tras E\vartheta + \delta v\tras C\vartheta   ] \dtt.
\end{split}
\eenl
The estimates in Lemmas \ref{lemmaxbar}, \ref{lemmadeltax} and \ref{lemmaeta} yield
$\Delta \Omega = \intT \delta v\tras C\vartheta  \dtt + o(\gamma).$
Integrating by parts in the latter expression and using \eqref{esteta} leads to
\benl
\intT \delta v\tras C\vartheta  \dtt 
=
[\yb\tras C\vartheta]_0^T - \intT \yb\tras(\dot {C}\vartheta + C \dot\vartheta  )\dtt
= o(\gamma), 
\eenl
and hence the desired result follows.
\end{proof}

\bibliographystyle{plain}
\bibliography{../../BIB/shooting,../../BIB/singulararc,../../BIB/hjb,../../BIB/diopt,../../BIB/generalsoledad}

\end{document}